\documentclass[10pt, leqno]{amsart}
\usepackage{amsfonts,amssymb}
\usepackage{amscd}

\usepackage{amsmath}
\usepackage{amsthm}
\usepackage{amsrefs}
\usepackage{qsymbols}
\usepackage{latexsym}
\usepackage{graphicx}
\usepackage{chngcntr}
\usepackage[noadjust]{cite}
\usepackage{paralist}
\usepackage{esint}

\newtheorem{theorem}{Theorem}[section]
\newtheorem{maintheorem}[theorem]{Main Theorem}
\newtheorem{lemma}[theorem]{Lemma}
\newtheorem{proposition}[theorem]{Proposition}
\newtheorem{corollary}[theorem]{Corollary}
\theoremstyle{definition}
\newtheorem{definition}[theorem]{Definition}
\newtheorem{remark}[theorem]{Remark}
\newtheorem{assumption}[theorem]{Assumption}

\newcommand{\IR}{\mathbb{R}}
\newcommand{\IC}{\mathbb{C}}
\newcommand{\IN}{\mathbb{N}}

\newcommand{\M}{\mathcal{M}}
\newcommand{\Lop}{{\mathcal{L}}}

\newcommand{\loc}{\mathrm{loc}}
\newcommand{\V}{\mathcal{V}}
\renewcommand{\a}{\mathfrak{a}}
\renewcommand{\j}{\mathfrak{j}}
\newcommand{\J}{\mathfrak{J}}
\newcommand{\SysA}{\mathbb{A}}
\newcommand{\SysV}{\mathbb{V}}
\renewcommand{\L}{\mathrm{L}}
\renewcommand{\H}{\mathrm{H}}
\newcommand{\F}{\mathrm{F}}
\newcommand{\C}{\mathrm{C}}
\newcommand{\m}{{\mathrm{m}}}
\newcommand{\R}{{\mathfrak{R}}}
\newcommand{\E}{{\mathfrak{E}}}
\renewcommand{\P}{{\mathfrak{P}}}
\newcommand{\f}{{\mathbf{f}}}
\newcommand{\g}{{\mathbf{g}}}
\newcommand{\bm}{{\mathbf{m}}}
\newcommand{\bdist}{{\mathbf{d}}}
\newcommand{\x}{{\mathbf{x}}}
\newcommand{\y}{{\mathbf{y}}}
\newcommand{\B}{{\mathbf{B}}}
\newcommand{\Dcyl}{\Omega \hspace{-2pt} \uparrow \hspace{-2.5pt} D}
\renewcommand{\d}{\mathrm{d}}

\newcommand{\eps}{\varepsilon}
\newcommand{\ind}{{\mathbf{1}}}
\newcommand{\argdot}{\, \cdot \,}
\newcommand{\abs}[1]{\lvert#1\rvert}
\newcommand{\cl}[1]{\overline{#1}}
\newcommand{\bd}{\partial}
\DeclareMathOperator{\supp}{supp}
\DeclareMathOperator{\dist}{d}
\DeclareMathOperator{\diam}{diam}
\DeclareMathOperator{\sgn}{sgn}
\renewcommand\Re{\operatorname{Re}}

\DeclareMathOperator{\Id}{Id}
\DeclareMathOperator{\Rg}{Rg}
\DeclareMathOperator{\Ke}{ker}
\DeclareMathOperator{\dom}{D}

\numberwithin{equation}{section}

\title{The Kato Square Root Problem for Mixed Boundary Conditions}
\author{Moritz Egert, Robert Haller-Dintelmann, and Patrick Tolksdorf}
\address{Fachbereich Mathematik, Technische Universit\"at Darmstadt, Schlossgartenstr. 7, 64289 Darmstadt, Germany}
\email{egert@mathematik.tu-darmstadt.de}
\email{haller@mathematik.tu-darmstadt.de}
\email{tolksdorf@mathematik.tu-darmstadt.de}
\thanks{The first and the third author were supported by ``Studienstiftung des deutschen Volkes''.}
\date{November 18, 2014}
\dedicatory{}

\begin{document}
\begin{abstract}
 We consider the negative Laplacian subject to mixed boundary conditions on a bounded domain. We prove under very general geometric assumptions that slightly above the critical exponent $\frac{1}{2}$ its fractional power domains still coincide with suitable Sobolev spaces of optimal regularity. In combination with a reduction theorem recently obtained by the authors, this solves the Kato Square Root Problem for elliptic second order operators and systems in divergence form under the same geometric assumptions. Thereby we answer a question posed by J.\ L.\ Lions in 1962 \cite{Counterexample_Lions}.
\end{abstract}
\maketitle

\section{Introduction}
\label{Sec: Introduction}

\noindent Let $-\nabla \cdot \mu \nabla$ be an elliptic differential operator in divergence form with bounded complex coefficients on a domain $\Omega$, subject to Dirichlet boundary conditions on some closed subset $D$ of the boundary $\bd \Omega$ and natural boundary conditions on $\bd \Omega \setminus D$ in the sense of the form method. Let $A$ be the maximal accretive realization of $-\nabla \cdot \mu \nabla$ on $\L^2(\Omega)$. The \emph{Kato Square Root Problem} for $A$ amounts to identifying the domain of the maximal accretive square root of $A$ as the domain of the corresponding form, i.e.\ the subspace of the first order Sobolev space $\H^1(\Omega)$ whose elements vanish on $D$. In this case $A$ is said to have the \emph{square root property}.

Whereas for self-adjoint $A$ the square root property is immediate from abstract form theory \cite{Kato}, the problem for non self-adjoint operators remained open for almost 40 years. For a historical survey explaining also the special role of the square root of $A$ compared to other fractional powers, we refer to \cite{McIntosh-KatoSurvey, Kato-Square-Root-Proof}. Shortly after being solved on the whole space by Auscher, Hofmann, Lacey, M\textsuperscript{c}Intosh, and Tchamitchian \cite{Kato-Square-Root-Proof}, \cite{Kato-Systems-Proof}, Auscher and Tchamitchian used localization techniques to give a proof on strongly Lipschitz domains complemented by either pure Dirichlet or pure Neumann boundary conditions \cite{Kato-homogeneousBoundary}. Earlier efforts concerning mixed boundary conditions culminated in the work of Axelsson, Keith, and M\textsuperscript{c}Intosh \cite{AKM}, who gave a proof for smooth domains with a Dirichlet part whose complement within the boundary is smooth and in addition -- due to the first order structure of the problem -- for global bi-Lipschitz images of these configurations.

The purpose of the present paper is to solve the Kato Square Root Problem on bounded domains under much more general geometric assumptions than in \cite{Kato-homogeneousBoundary} and \cite{AKM}. First and foremost we can dispense with the Lipschitz property of $\Omega$ in the following spirit: We assume that $\bd \Omega$ decomposes into a closed subset $D$, to be understood as the Dirichlet part, and its complement, which are allowed to share a common frontier. We demand that $D$ is a $(d-1)$-set in the sense of Jonsson-Wallin, or equivalently satisfies the Ahlfors-David condition, and only around $\cl{\bd \Omega \setminus D}$ do we demand local bi-Lipschitz charts. In addition, we in essence impose a plumpness, or interior corkscrew condition on $\Omega$, which, roughly speaking, excludes outward cusps also along the Dirichlet part. For precise definitions we refer to Section~\ref{Sec: Notation and general assumptions}.

In particular, $\Omega$ may be sliced or touch its boundary from two sides, see Figure~\ref{Fig: Cone} for a striking constellation.
\begin{figure}[ht]
\label{Fig: Cone}
	\centering
  \includegraphics[scale=0.4]{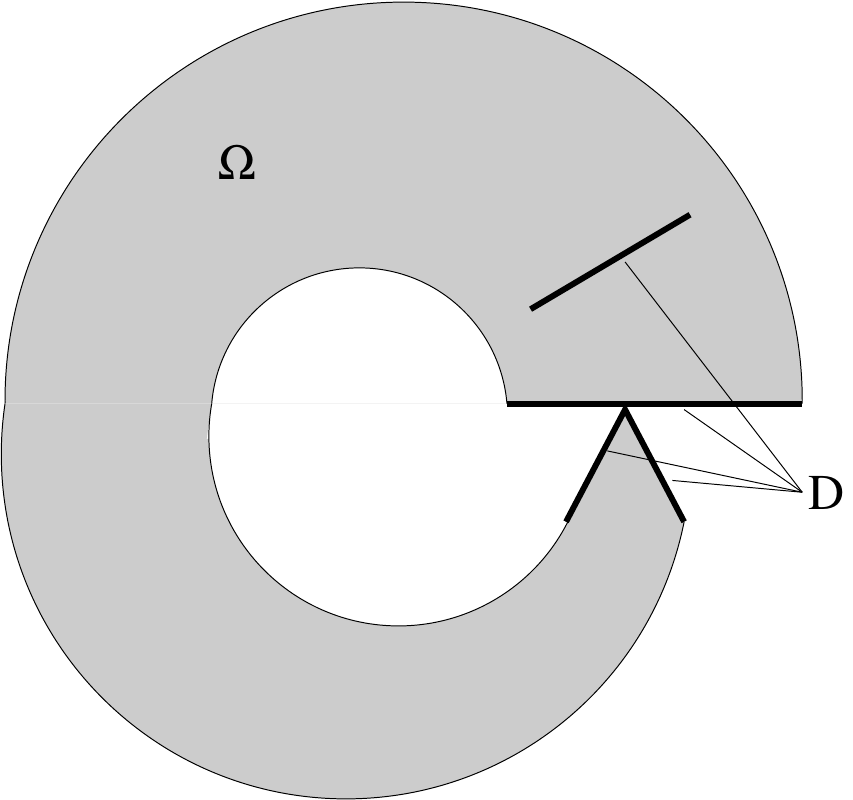}
	\caption{The domain $\Omega \subseteq \IR^2$ is obtained by smoothly deforming an acute triangle such that one apex touches the opposed side from outside. Afterwards a closed line segment is removed from its interior. Around the points on this line segment, as well as around the former apex, the Lipschitz condition for $\partial \Omega$ is violated as $\Omega$ does not locally lie on one side of its boundary -- but these parts belong to the Dirichlet part $D$. Around $\overline{\partial \Omega \setminus D}$ the boundary of $\Omega$ is smooth and since $D$ is a union of Lipschitz curves, it satisfies the Ahlfors-David condition.  }
	\label{fig1}
\end{figure}

As special cases the pure Dirichlet ($D = \bd \Omega$) and the pure Neumann case ($D = \emptyset$) are included. Let us stress that in the former we can dispense with the Lipschitz property of the domain completely.

More recently, relative results including the square root property for $A$ as an assumption have been obtained. This concerns extrapolation of the square root property to $\L^p$ spaces \cite{ABHR}, maximal parabolic regularity on distribution spaces \cite{ABHR}, and perturbation theory \cite{Gesztesy-Potentials}. One of our main motivations for writing the present paper was to close this gap between geometric constellations in which the Kato Square Root Problem is solved and those in which its solution already applies to other topics.

It is convenient to view the Kato Square Root Problem as the problem of proving optimal Sobolev regularity for the domain of the square root of $A$. Indeed, as $A$ is associated to a second order differential operator, the domain of $A$ allows for at most two distributional derivatives in $\L^2(\Omega)$. Hence, by interpolation the optimal Sobolev regularity for the domain of its square root is one distributional derivative in $\L^2(\Omega)$. It is remarkable that optimal Sobolev regularity for the domain can even fail for the negative Laplacian if $\Omega$ is smooth \cite{Shamir-Counterexample}, whereas in this case optimal regularity for the domain of the square root is immediate by self-adjointness.

Recently we proved that the Kato Square Root Problem for any elliptic operator in divergence form on $\Omega$, complemented with local homogeneous boundary conditions, can be reduced to a regularity result for the fractional powers of the simplest operator in this class -- the negative Laplacian \cite{Laplace-Extrapolation-Implies-Kato}. In essence, it has to be shown that there exists an $\alpha > \frac{1}{2}$ such that the domain of $(-\Delta)^\alpha$ is a Sobolev space of optimal order $2 \alpha$, see Section~\ref{Sec: Main results} for details. This should be regarded as the extrapolation of the square root property for $-\Delta$, which refers to the case $\alpha = \frac{1}{2}$. 

As our main theorem we prove this extrapolation result for the negative Laplacian in the described geometric setting, thereby solving the Kato Square Root Problem via reduction to the results in \cite{Laplace-Extrapolation-Implies-Kato}. In the case of a real coefficient matrix $\mu$ this also yields the solution to the Square Root Problem for mixed boundary conditions on $\L^p(\Omega)$ for $p \in (1,2)$, cf.\ \cite{ABHR}.

The paper is organized as follows. In Section~\ref{Sec: Notation and general assumptions} we introduce some general notation, fix our geometric setting and properly define the elliptic operator under consideration. In Section~\ref{Sec: Sobolev mixed boundary} we introduce a continuous scale $\{\H_D^s(\Omega)\}_{1/2 < s < 3/2}$ of $\L^2$ based Sobolev spaces related to mixed boundary conditions and establish some preliminary properties. Subsequently, we state our main result in Section~\ref{Sec: Main results} and infer from it the solution to the Kato Square Root Problem. The proof of our main result is presented later on in Section~\ref{Sec: Proof of the main result}. Our proof is based on an interpolation argument going back to Pryde \cite{Pryde-MixedBoundary}. The same idea has been utilized in \cite{AKM}. 

Due to the generality of our geometric setting -- in particular because localization techniques are not feasible around the Dirichlet part of the boundary -- the adaption of Pryde's argument requires some preparations. These lead to new results that are interesting in themselves. We develop a suitable interpolation theory for the family $\{\H_D^s(\Omega)\}_{1/2 < s < 3/2}$ in Section~\ref{Sec: Interpolation for HDs} relying on two key ingredients. Firstly, in Section~\ref{Sec: Extension operators} we construct a degree independent extension operator, heavily resting on Rogers' universal extension operator for $(\eps,\delta)$-domains \cite{Rogers-Article} and recent results on fractional Hardy inequalities \cite{Hardy-Triebel-LizorkinII}, \cite{Edmunds-Hardy}, \cite{Dyda-Vahakangas}, \cite{Hardy-Triebel-LizorkinI}. Secondly, we prove a fractional Hardy type inequality for Sobolev spaces with partially vanishing boundary trace in Section~\ref{Sec: Fractional Hardy for HDs}, thereby extending a result from \cite{ABHR}.

Finally, in Section~\ref{Sec: Elliptic Systems} we extend our proof of the Kato Square Root Problem to coupled elliptic systems. Here, we can even allow for a different Dirichlet part for each component.

\subsection*{Acknowledgments}

The authors want to thank J.\ Rehberg for valuable discussions and hints on the topic, and A.\ V.\ V\"ah\"akangas for generously providing us with a preprint of \cite{Edmunds-Hardy} and giving the decisive hint for the proof of Proposition~\ref{Prop: Hardy for testfunctions on Omega-bullet}.

\section{Notation and General Assumptions}
\label{Sec: Notation and general assumptions}

\noindent Most of our notation is standard. Throughout the dimension $d \geq 2$ of the surrounding Euclidean space is fixed. The open ball in $\IR^d$ with center $x$ and radius $r$ is denoted by $B(x,r)$. The symbol $\abs{\argdot}$ is used for the absolute value of complex numbers and the Euclidean norm of vectors in $\IR^d$ as well as for the $d$-dimensional Lebesgue measure. The Euclidean distance between subsets $E$ and $F$ of $\IR^d$ is $\d(E,F):= \inf \{\abs{x-y}: x \in E, \, y \in F \}$. If $E = \{x\}$ then the abbreviation $\d_F(x)$ is used. By a domain we always mean an open connected set.

Integration is with respect to the Lebesgue measure on $\IR^d$ unless the contrary is claimed. The same applies to measure theoretic abbreviations such as \emph{a.e.}\ (almost everywhere) and \emph{a.a.}\ (almost all). For average integrals the symbol $\fint$ is used. The Lebesgue spaces on a complete measure space $(X, \Sigma, \mu)$ are denoted by $\L^p(X, \mu)$.

Any Banach space $X$ under consideration is taken over the complex numbers. Its norm is usually denoted by $\|\cdot\|_{X}$. If $Y$ is another Banach space then $X = Y$ means that $X$ and $Y$ coincide as sets and that their norms are equivalent. The domain of a closed operator $B$ on $X$ is denoted by $\dom(B)$. It is usually regarded as a Banach space equipped with the graph norm. The space of bounded linear operators from $X$ to $Y$ is $\Lop(X,Y)$ and $\Lop(X,X)$ is abbreviated by $\Lop(X)$.

We will use the generic constants convention and write $\lesssim$ and $\gtrsim$ for inequalities that hold up to multiplication by a constant $C>0$ not depending on the parameters at stake. We write $a \simeq b$ if $a \lesssim b$ and $b \lesssim a$ hold.

\subsection{Geometric setting}
\label{Subsec: Geometric Setting}

Let us state precisely our geometric assumptions concerning the domain $\Omega$ and the Dirichlet part $D$ of its boundary. First, we recall the notion of an $l$-set according to Jonsson-Wallin \cite[Sec.\ VIII.1.1]{Jonsson-Wallin}.

\begin{definition}
\label{Def: (d-1)-set}
Assume $0 < l \leq d$. A non-empty Borel set $F \subseteq \IR^d$ is called \emph{$l$-set} if 
\begin{align*}
 \m_l(F \cap B(x,r)) \simeq r^l \qquad (x \in F,\, 0 < r \leq 1),
\end{align*}
where here and henceforth $\m_l$ denotes the $l$-dimensional Hausdorff measure on $\IR^d$.
\end{definition}

\begin{remark}
\label{Rem: Remarks on d-1 set}
\begin{enumerate}
\itemsep3pt
 \item The condition $r \leq 1$ can be replaced by $r \leq r_0$ for any fixed $r_0 > 0$. Also one can replace open balls by closed ones without changing the notion of an $l$-set.
 \item If $F$ is an $l$-set then so is $\cl{F}$ and $\cl{F} \setminus F$ has $\m_l$-measure zero \cite[Sec.\ VIII.1.1]{Jonsson-Wallin}. At many occasions this allows us to assume without loss of generality that a given $l$-set is closed.
 \item An equivalent, commonly used notion for $(d-1)$-sets is that of sets satisfying the \emph{Ahlfors-David condition}.
 \item We will occasionally use that the union of two $l$-sets $E,F \subseteq \IR^d$ is a again an $l$-set. To see this, fix $x \in E \cup F$ and $r \in (0,1]$. Without restrictions assume $x \in E$. If $F \cap B(x,r) = \emptyset$ then $\m_l((E \cup F) \cap B(x,r)) \simeq r^l$ is immediate. Otherwise there exists some $y \in F \cap B(x,r)$ and the assertion follows from the inclusions
\begin{align*}
 E \cap B(x,r) \subseteq (E \cup F) \cap B(x,r) \subseteq (E \cap B(x,r)) \cup (F \cap B(y,2r)).
\end{align*}
\end{enumerate}

\end{remark}

Throughout this work we suppose the following configuration.

\begin{assumption}
\label{Ass: General geometric assumption on Omega}
\begin{enumerate}
\itemsep3pt
 \item The domain $\Omega \subseteq \IR^d$, $d \geq 2$, is non-empty, bounded, and $D$ is a closed subset of its boundary $\bd \Omega$. The domain $\Omega$ itself is a $d$-set, i.e.\
 \begin{align*}
  \abs{\Omega \cap B(x,r)} \simeq r^d \qquad (x \in \Omega, 0< r \leq 1)
 \end{align*}
 since the $d$-dimensional Hausdorff measure on $\IR^d$ is equivalent to the Lebesgue measure.
 \item The set $D$, to be understood as the \emph{Dirichlet part} of $\bd \Omega$, is either empty or a $(d-1)$-set.
 \item For every $x \in \cl{\bd \Omega \setminus D}$ there exists an open neighborhood $U_x$ and a bi-Lipschitz map $\Phi_x$ from $U_x$ onto the unit cube $(-1,1)^d$ such that
  \begin{align*}
  \Phi_x(x) &= 0, \\
  \Phi_x(\Omega \cap U_x) &= (-1,1)^{d-1} \times (-1,0), \\
  \Phi_x(\bd \Omega \cap U_x) &= (-1,1)^{d-1} \times \{0\}.
  \end{align*}
\end{enumerate}
\end{assumption}

\begin{remark}
\label{Rem: Remark on geometric assumptions}
\begin{enumerate}
\itemsep3pt
 \item Under Assumption~\ref{Ass: General geometric assumption on Omega} the boundary of $\Omega$ is a $(d-1)$-set as well. Indeed, thanks to the bi-Lipschitz parametrizations for each $x \in \cl{\bd \Omega \setminus D}$ the set $\bd \Omega \cap U_x$ is a $(d-1)$-set. Hence, by compactness of $\cl{\bd \Omega \setminus D}$, the boundary of $\Omega$ can be written as the finite union of $(d-1)$-sets and the claim follows from Remark~\ref{Rem: Remarks on d-1 set}.
 \item Being a bounded open $d$-set whose boundary is a $(d-1)$-set, the domain $\Omega \subseteq \IR^d$ also satisfies the following \emph{plumpness} or \emph{interior corkscrew} condition: There exists a $\kappa \in (0,1)$ with the property that for each $x \in \cl{\Omega}$ and each $r \in (0, \diam(\Omega))$ there exists $y \in \cl{B(x,r)}$ such that $\cl{B(y, \kappa r)} \subseteq \Omega$, cf.\ \cite[Cor.~7.6]{Laplace-Extrapolation-Implies-Kato}. 

 Of course the plumpness condition is a stronger condition than being a $d$-set. From a geometric point of view it prevents $\Omega$ from having outward cusps. It is easy to see that every bounded Lipschitz domain is plump. For a domain that fulfills Assumption~\ref{Ass: General geometric assumption on Omega} but notably violates the Lipschitz property see Figure~\ref{Fig: Cone}. 
 \item Part (3) of Assumption~\ref{Ass: General geometric assumption on Omega} is void in the case of pure Dirichlet boundary conditions, i.e.\ if $D = \bd \Omega$. Hence, in this case we can dispense with the Lipschitz property of the boundary completely.
\end{enumerate}
\end{remark}

\subsection{The elliptic operator}
\label{Subsec: The Elliptic Operator}

Next, we define the elliptic operator $- \nabla \cdot \mu \nabla$ under consideration properly by means of Kato's form method \cite{Kato}. We begin with the following sets of test functions.

\begin{definition}
\label{Def: CFinfty}
If $\Xi \subseteq \IR^d$ is a domain and $F$ a subset of $\cl{\Xi}$ put
\begin{align}
\label{Eq: CFinfty}
 \C_F^\infty(\Xi):= \left \{u|_\Xi: u \in \C_c^\infty(\IR^d),\, \dist(\supp(u), F) > 0 \right \}.
\end{align}
\end{definition}

\begin{remark}
\label{Rem: CFinfty compact support only on unbounded domains}
The assumption that $u$ in \eqref{Eq: CFinfty} has compact support can be dropped if $\Xi$ is bounded. Hence, Definition~\ref{Def: CFinfty} is in accordance with \cite{ABHR}, \cite{Haller-Rehberg}.
\end{remark}

The form domain $\V$ that incorporates Dirichlet boundary conditions on $D$ and natural boundary conditions on $\partial \Omega \setminus D$ is defined in the usual way, see e.g.\ \cite{Ouhabaz}, as follows. The reader should also compare this definition of the form domain to the more restrictive one used in \cite{AKM}.

\begin{definition}
\label{Def: Form domain}
The form domain $\V$ is the closure of $\C_D^\infty(\Omega)$ under the Hilbertian norm
\begin{align*}
 \|u\|_{\V} := \bigg(\int_\Omega \abs{u}^2 + \abs{\nabla u}^2\bigg)^{1/2} \qquad (u \in \V).
\end{align*}
\end{definition}

\begin{remark}
\label{Rem: Other characterization of form domain}
In Section~\ref{Sec: Main results} we will give a different characterization of $\V$ as the subspace of the first order Sobolev space $\H^1(\Omega)$ whose elements vanish on $D$ in the sense of a trace.
\end{remark}

Next, we state our assumptions on the coefficient matrix.

\begin{assumption}
\label{Ass: Ellipticity}
The coefficient matrix $\mu$ is a Lebesgue measurable, bounded function on $\Omega$ taking its values in the set of complex $d \times d$ matrices. The associated sesquilinear form
\begin{align*}
 \a: \V \times \V \to \IC, \quad \a(u,v) = \int_\Omega \mu \nabla u \cdot \nabla \cl{v}
\end{align*}
is elliptic in the sense that for some $\lambda > 0$ it satisfies the \emph{G\aa{}rding inequality}
\begin{align*}
 \Re (\a(u,u)) \geq \lambda \|\nabla u\|_{\L^2(\Omega; \IC^d)}^2 \qquad (u \in \V).
\end{align*}
\end{assumption}

Since $\V$ is dense in $\L^2(\Omega)$ and $\a$ is elliptic, it is known from classical form theory, see e.g.\ \cite[Sec.~VI]{Kato} that there exists a unique maximal accretive operator $A$ on $\L^2(\Omega)$ such that $\dom(A) \subseteq \V$ and
\begin{align*}
  \a(u,v) = \langle Au, v\rangle_{\L^2(\Omega)} \qquad (u \in \dom(A),\, v \in \V).
\end{align*}
Here, an operator $B$ on a Hilbert space $H$ is \emph{maximal accretive} if it is closed, and if for $z$ in the open right complex halfplane $z + B$ is invertible with $\|(z + B)^{-1}\|_{\Lop(H)} \leq \Re (z)^{-1}$. 

As usual, the divergence form operator $- \nabla \cdot \mu \nabla$ is identified with $A$. Note that if $\mu$ is the identity matrix then $-A$ is the \emph{weak Laplacian} with form domain $\V$, denoted by $\Delta_\V$ in the following. More precisely, this operator is characterized by
\begin{align*}
 \int_\Omega \nabla u \cdot \nabla \cl{v} = - \langle \Delta_\V u, v\rangle_{\L^2(\Omega)} \qquad (u \in \dom(\Delta_\V),\, v \in \V).
\end{align*}

The \emph{fractional powers} $A^\alpha$, $\Re(\alpha) > 0$, can be defined by the functional calculus for sectorial operators, see e.g.\ \cite{Haase}, \cite{Lunardi-Interpolation}. They are closed operators given by the \emph{Balakrishnan Representation}
\begin{align}
\label{Balakrishnan Representation}
 A^\alpha u = \frac{\Gamma(k)}{\Gamma(\alpha) \Gamma(k-\alpha)} \int_0^\infty t^{\alpha -1} (A(t+A)^{-1})^k u \; \d t
\end{align}
if $k > \Re(\alpha)$ is an integer and $u \in \dom(A^k)$. In this case $\dom(A^k)$ is a core for $\dom(A^\alpha)$. Given $\eps > 0$, the fractional powers of the maximal accretive operator $\eps + A$ are defined analogously and their domains satisfy $\dom((\eps + A)^\alpha) = \dom(A^\alpha)$, $\Re(\alpha)>0$, with equivalence of the respective graph norms. For proofs the reader may consult \cite[Sec.\ ~3.1]{Haase}.

\section{Sobolev Spaces related to Mixed Boundary Conditions}
\label{Sec: Sobolev mixed boundary}

\noindent We introduce a continuous scale of Sobolev spaces related to mixed boundary conditions and establish some preliminary properties that will be needed later on. Unless the contrary is claimed, all function spaces are of complex valued functions.

\subsection{Sobolev spaces on $\IR^d$ and on domains}
\label{Subsec: Sobolev spaces}

For $s \in \IN_0$ denote by $\H^s(\IR^d)$ the Sobolev space of $\L^2(\IR^d)$ functions whose distributional derivatives up to order $s$ are in $\L^2(\IR^d)$ equipped with the usual Hilbert space norm. With this convention, $\H^0(\IR^d) = \L^2(\IR^d)$. For $s \in \IR_+ \setminus \IN_0$ let $k:= \lfloor s \rfloor$ be the integer part of $s$ and set $\theta:= s-k$. The respective (fractional) Sobolev space then is
\begin{align*}
 \H^{s}(\IR^d) := \Big \{f \in \H^k(\IR^d): \|f\|_{\H^{s}(\IR^d)}:= \|f\|_{\H^k(\IR^d)} + \sum_{\abs{\alpha}= k} [\partial^\alpha f]_\theta < \infty \Big \},
\end{align*}
where
\begin{align*}
 [f]_\theta := \bigg( \iint_{\abs{x-y} < 1}\frac{\abs{f(x)-f(y)}^2}{\abs{x-y}^{d+2\theta}} \; \d x \; \d y \bigg)^{1/2}.
\end{align*}
The definition of $[\argdot]_\theta$ above differs from the common ones, cf.\ \cite[pp.\ 189-190]{Triebel}, as integration does not take place over all of $\IR^d \times \IR^d$. This is due to technical reasons and one readily checks that the different definitions lead to equivalent norms on $\H^{s}(\IR^d)$. 

Without further mentioning we will frequently use that for $s \in \IR_+$ the Sobolev space $\H^s(\IR^d)$ coincides with both the Bessel potential space $\H_2^s(\IR^d)$ and the Triebel-Lizorkin space $\F_{2,2}^s(\IR^d)$, see e.g.\ \cite[pp.\ 172, 177, 189-190]{Triebel}. Also recall that for each $s \in \IR_+$ the set $\C_c^\infty(\IR^d)$ of smooth functions with compact support is dense in $\H^s(\IR^d)$, see e.g.\ \cite[Sec.\ 2.3.2]{Triebel}.

For $\Xi \subseteq \IR^d$ a domain and $s \in \IR_+$ the Banach space $\H^s(\Xi)$ is defined by restricting functions from $\H^s(\IR^d)$ to $\Xi$, i.e.\ by $\H^s(\Xi):= \{f|_{\Xi}: f \in \H^s(\IR^d)\}$ equipped with the usual quotient norm 
\begin{align*}
 \|f\|_{\H^s(\Xi)} := \inf \left \{\|g\|_{\H^s(\IR^d)}: g|_{\Xi} = f \right\} \qquad (f \in \H^s(\Xi)).
\end{align*}

\begin{remark}
\label{Rem: Sobolev spaces not locally defined}
Let $\Xi$, $s$, $k$, and $\theta$ be as before. Note carefully that by construction
\begin{align*}
  \bigg(\sum_{\abs{\alpha} \leq k} \|\partial^\alpha f\|_{\L^2(\Xi)}^2\bigg)^{1/2} + \sum_{\abs{\alpha}= k} \bigg( \iint_{\substack{x,y \in \Xi \\ \abs{x-y} < 1}}\frac{\abs{\partial^\alpha f(x)-\partial^\alpha f(y)}^2}{\abs{x-y}^{d+2\theta}} \; \d x \; \d y \bigg)^{1/2} \leq \|f\|_{\H^{s}(\Xi)}
\end{align*}
holds for all $f \in \H^s(\Xi)$, where we think of the second term as not being present if $s \in \IN_0$, but without further assumptions on $\Xi$ these norms are not comparable, cf.\ \cite[Sec.\ 5]{Sobolev-Hitchhiker}.
\end{remark}

\subsection{Sobolev spaces with partially vanishing traces}
\label{Subsec: Sobolev spaces with partially vanishing traces}

For the Sobolev spaces with partially vanishing boundary traces we restrict ourselves to $s \in (\frac{1}{2},\frac{3}{2})$ since only these values will be of interest in the following.

Fractional Sobolev spaces on $(d-1)$-sets can be defined in a natural way as long as $s \in (0,1)$. We follow the presentation in \cite{Jonsson-Wallin} but for consistency stick to the notation $\H^s$ rather than $\mathrm{B}_{2,2}^s$.

\begin{definition}
\label{Def: Sobolev spaces on (d-1)-sets}
Let $F \subseteq \IR^d$ be a $(d-1)$-set and $s \in (0,1)$. The fractional Sobolev space $\H^s(F)$ consists of those $f \in \L^2(F,\, \m_{d-1})$ that satisfy
\begin{align*}
 \|f\|_{\H^s(F)}:= \|f\|_{\L^2(F,\, \m_{d-1})} + \bigg( \iint_{\substack{x,y \in F \\ \abs{x-y} < 1}}\frac{\abs{f(x)-f(y)}^2}{\abs{x-y}^{d-1+2s}} \; \d \m_{d-1}(x) \; \d \m_{d-1}(y) \bigg)^{1/2} < \infty.
\end{align*}
Equipped with the norm $\|\cdot\|_{\H^s(F)}$ it becomes a Banach space.
\end{definition}

The ultimate instrument for the treatise of Sobolev spaces with partially vanishing boundary traces is the following extension-restriction result. We refer to Sections VII.1.1 and VII.2.1 in \cite{Jonsson-Wallin} for the first two assertions and to \cite[Thm.\ 2.5]{Rehberg-Jonsson} for the third.

\begin{proposition}
\label{Prop: Jonsson-Wallin trace theorem}
Let $F \subseteq \IR^d$ be a $(d-1)$-set and $s \in (\frac{1}{2},\frac{3}{2})$.
\begin{enumerate}
\itemsep3pt
 \item For $f \in \H^s(\IR^d)$ the limit
  \begin{align*}
   (\R_F f)(x_0):= \lim_{r \to 0} \fint_{B(x_0,r)} f(x) \; \d x
  \end{align*}
 exists for $\m_{d-1}$-almost all $x_0 \in F$. The so defined restriction operator $\R_F$ maps $\H^s(\IR^d)$ boundedly onto $\H^{s-1/2}(F)$.
 \item Conversely, there exists a bounded extension operator $\E_F: \H^{s-1/2}(F) \to \H^s(\IR^d)$ which forms a right inverse for $\R_F$. By construction $\E_F$ does not depend on $s$.
 \item The operator $\E_F$ maps Lipschitz continuous functions on $F$ to Lipschitz continuous functions on $\IR^d$.
\end{enumerate}
\end{proposition}

\begin{definition}
\label{Def: Sobolev spaces with vanishing trace}
Let $F \subseteq \IR^d$ be a $(d-1)$-set, $s \in (\frac{1}{2},\frac{3}{2})$, and $\R_F$ as in Proposition~\ref{Prop: Jonsson-Wallin trace theorem}. 
\begin{enumerate}
\itemsep3pt
 \item Put
  \begin{align*}
   \H_F^s(\IR^d):= \big \{f \in \H^s(\IR^d): \R_F f = 0 \quad \text{$\m_{d-1}$-a.e.\ on $F$} \big \},
  \end{align*}
 which by continuity of $\R_F$ is a closed subspace of $\H^s(\IR^d)$ and thus complete under the inherited norm. It is convenient to also define $\H_{\emptyset}^s(\IR^d) := \H^s(\IR^d)$.
 \item If $\Xi \subseteq \IR^d$ is a domain and $F \subseteq \cl{\Xi}$ put $\H_F^s(\Xi):= \{f|_{\Xi}: f \in \H_F^s(\IR^d)\}$ and equip it with the usual quotient norm. Again, also define $\H_{\emptyset}^s(\Xi) := \H^s(\Xi)$.
\end{enumerate}
\end{definition}

\begin{corollary}
\label{Cor: HFs is complemented} 
If $F \subseteq \IR^d$ is a $(d-1)$-set and $s \in (\frac{1}{2},\frac{3}{2})$ then $\H_F^s(\IR^d)$ is a complemented subspace of $\H^s(\IR^d)$ with bounded projection $\P_F:= \Id - \E_F \R_F$. 
\end{corollary}

\begin{proof}
The right inverse property $\R_F \E_F = \Id$ on $\H^{s-1/2}(F)$, see Proposition~\ref{Prop: Jonsson-Wallin trace theorem}, immediately implies $\P_F^2 = \P_F$. Moreover, $f \in \H^s(\IR^d)$ satisfies $\P_F f = f$ if and only if $\E_F \R_F f = 0$ holds. Again by the right inverse property the latter is equivalent to $\R_F f = 0$, i.e.\ to $f\in \H_F^s(\IR^d)$.
\end{proof}

In the setting of Definition~\ref{Def: Sobolev spaces with vanishing trace} we think of $\R_F$ as the pointwise restriction and of $\H_F^s(\Xi)$ as the subspace of $\H^s(\Xi)$ containing the functions that vanish on $F$. If $f \in \H^s(\IR^d)$ is continuous (i.e.\ has a continuous representative) then the limit defining $\R_F f$ exists for all $x_0 \in F$ and indeed coincides with the pointwise restriction of (the continuous representative of) $f$ to $F$.

The following lemma on multiplication operators will be needed later on.

\begin{lemma}
\label{Lem: Multiplication trace lemma}
Let $\Xi \subseteq \IR^d$ be a domain and let $\eta: \IR^d \to \IC$ be bounded and twice differentiable with bounded derivatives up to order two.
\begin{enumerate}
\itemsep3pt
 \item The multiplication operator $M_\eta$ associated to $\eta$ is bounded on $\H^s(\Xi)$ if $s \in [0,2]$.
 
 \item Assume that $E\subseteq \cl{\Xi}$ is a $(d-1)$-set and that $F \subseteq E$ is either empty or a $(d-1)$-set. If $\eta$ vanishes on $E \setminus F$ then $M_\eta$ maps $\H_F^s(\Xi)$ boundedly into $\H_{E}^s(\Xi)$ for each $s \in (\frac{1}{2}, \frac{3}{2})$.
\end{enumerate}
\end{lemma}

\begin{proof}
For the first claim let $s \in [0, 2]$. Since $M_\eta$ is bounded on $\L^2(\IR^d)$ and on $\H^2(\IR^d)$ its boundedness on $\H^s(\IR^d)$ follows by complex interpolation, see e.g.\ \cite[Thm.\ 6.4.5]{Bergh-Loefstroem}. Boundedness on $\H^s(\Xi)$ then is immediate from the definition of the quotient norm.

For the second claim let $s \in (\frac{1}{2}, \frac{3}{2})$, fix $f \in \H_F^s(\Xi)$, and let $g \in \H_F^s(\IR^d)$ be an extension of $f$. Passing to the limit $r \to 0$, due to Proposition~\ref{Prop: Jonsson-Wallin trace theorem} the left-hand side of
\begin{align*}
 \fint_{B(x_0,r)} M_\eta g(x) \; \d x = \fint_{B(x_0,r)} g(x)(\eta(x) - \eta(x_0)) \; \d x  + \eta(x_0) \fint_{B(x_0,r)} g(x) \; \d x
\end{align*}
converges to $\R_E M_\eta g(x_0)$ for $\m_{d-1}$-almost all $x_0 \in E$ and, as a consequence of $g \in \H_F^s(\IR^d)$, the second term on the right-hand side tends to zero for $\m_{d-1}$-almost all $x_0 \in F$. Taking into account that $\eta$ vanishes on $E \setminus F$ it follows for $\m_{d-1}$-almost all $x_0 \in E$ that
\begin{align}
\label{Splitting trace for multiplication: Eq1}
 \R_E M_\eta g(x_0) = \lim_{r \to 0} \fint_{B(x_0,r)} g(x)(\eta(x) - \eta(x_0)) \; \d x.
\end{align}
Note that $\R_E \abs{g}(x_0)$ is defined for $\m_{d-1}$-almost all $x_0 \in E$: Indeed, let $t \in (\frac{1}{2},1)$ be smaller than $s$. Then of course $g \in \H^t(\IR^d)$ and due to $t<1$ one can check $\abs{g} \in \H^t(\IR^d)$ by the reverse triangle inequality. 

If finally $x_0 \in E$ is such that the limit in \eqref{Splitting trace for multiplication: Eq1} exists and $\R_E \abs{g}(x_0)$ is defined then
\begin{align*}
 \abs{\R_E M_\eta g(x_0)} \leq \lim_{r \to 0} \|\eta- \eta(x_0)\|_{\L^\infty(B(x_0,r))} \fint_{B(x_0,r)} \abs{g(x)} \; \d x = 0
\end{align*}
by continuity of $\eta$. This proves $\R_E M_\eta g = 0$, i.e.\ $M_\eta g \in \H_E^s(\IR^d)$. Since $g$ was an arbitrary $\H_F^s(\IR^d)$ extension of $f$, the boundedness of $M_\eta: \H_F^s(\Xi) \to \H_E^s(\Xi)$ follows.
\end{proof}

For the following approximation result recall the spaces $\C_F^\infty$ from Definition~\ref{Def: CFinfty}.

\begin{proposition}
\label{Prop: Smooth functions dense in HFs}
Let $\Xi \subseteq \IR^d$ be a domain, let $F \subseteq \cl{\Xi}$ be either empty or a $(d-1)$-set, and let $s \in (\frac{1}{2},1]$. Then $\C_F^\infty(\Xi)$ is dense in $\H_F^s(\Xi)$.
\end{proposition}

\begin{proof}
The second part of Remark~\ref{Rem: Remarks on d-1 set} entails $\H_F^s(\Xi) = \H_{\cl{F}}^s(\Xi)$ so that without restrictions we can assume that $F$ is closed.

Obviously $\C_F^\infty(\Xi)$ is a subset of $\H_F^s(\Xi)$. To prove density, fix $f \in \H_F^s(\Xi)$ and choose an extension $g \in \H_F^s(\IR^d)$ of $f$. Let $(g_n)_n$ be a sequence from $\C_c^\infty(\IR^d)$ converging to $g$ in $\H^s(\IR^d)$. If $F = \emptyset$ then $(g_n|_{\Xi})_n \subseteq \C_F^{\infty}(\Xi)$ converges to $f$ in $\H_F^s(\Xi)$. So, for the rest of the proof assume that $F$ is a $(d-1)$-set and let $\P_F:\H^s(\IR^d) \to \H_F^s(\IR^d)$ be the projection introduced in Corollary~\ref{Cor: HFs is complemented}. Then $(\P_F g_n)_n$ converges to $\P_F g =g$ in $\H^s(\IR^d)$. Since $\H^1(\IR^d)$ continuously embeds into $\H^s(\IR^d)$ it suffices to show:
\begin{align}
\label{Smooth functions dense: Intermediate goal}
 \text{For every $n \in \IN$ there exists an $h_n \in \C_F^\infty(\IR^d)$ such that $\|h_n - \P_F g_n\|_{\H^1(\IR^d)} \leq \frac{1}{n}$.}
\end{align}
The sequence $(h_n|_{\Xi})_n$ then converges to $f$ in $\H_F^s(\Xi)$.

To establish \eqref{Smooth functions dense: Intermediate goal} fix $n \in \IN$ and note that by the third part of Proposition~\ref{Prop: Jonsson-Wallin trace theorem} the function $\P_F g_n$ has a Lipschitz continuous representative $\mathfrak{g}_n$ which by construction vanishes $\m_{d-1}$-a.e.\ on $F$. As $F$ is a $(d-1)$-set, the $\m_{d-1}$-measure of every non-empty relatively open subset of $F$ is strictly positive. Therefore $\mathfrak{g}_n$ must vanish everywhere on $F$. Now, a classical approximation result yields a function $h \in \H^1(\IR^d)$ with support in $\IR^d \setminus F$ such that $\|h-\P_F g_n\|_{\H^1(\IR^d)} \leq \frac{1}{2n}$. For a proof we refer to \cite[Thm.\ 9.1.3]{Adams-Hedberg} or to \cite[Sec.\ 9.2]{Adams-Hedberg} for an elementary argument that suffices in our case. To obtain $h_n \in \C_F^\infty(\IR^d)$ as in \eqref{Smooth functions dense: Intermediate goal} simply convolve $h$ with a smooth kernel with sufficiently small support (here the closedness of $F$ comes into play) and then multiply with a smooth cut-off function with sufficiently large support.
\end{proof}

\begin{corollary}
\label{Cor: Form domain equals HD1}
Up to equivalent norms $\V = \H_D^1(\Omega)$.
\end{corollary}

\begin{proof}
Under Assumption~\ref{Ass: General geometric assumption on Omega} there exists a bounded extension operator $\E_\V: \V \to \H_D^1(\IR^d)$ such that for every $f \in \V$ it holds $\E_\V f = f$ a.e.\ on $\Omega$, see \cite[Lem.~3.2]{ABHR}. Note that in \cite{ABHR} the space $\H_D^1(\IR^d)$ is defined as the completion of $\C_D^\infty(\IR^d)$ in the $\H^1(\IR^d)$ norm but in view of Proposition~\ref{Prop: Smooth functions dense in HFs} this definition coincides with ours. Now, each $f \in \C_D^\infty(\Omega)$ satisfies
\begin{align*}
 \|f\|_{\L^2(\Omega)}^2 + \|\nabla f \|_{\L^2(\Omega; \IC^d)}^2 
= \inf \left \{\|g\|_{\L^2(\Omega)}^2 + \|\nabla g \|_{\L^2(\Omega; \IC^d)}^2: g \in \H_D^1(\IR^d), \, g|_{\Omega} = f \right \} \leq \|f\|_{\H_D^1(\Omega)}^2,
\end{align*}
but since $f$ coincides with $\E_\V f \in \H_D^1(\IR^d)$ on $\Omega$ also
\begin{align*}
 \|f\|_{\H_D^1(\Omega)}^2 \leq \|\E_{\V}f\|_{\H_D^1(\IR^d)}^2 \lesssim \|f\|_{\V}^2 = \|f\|_{\L^2(\Omega)}^2 + \|\nabla f \|_{\L^2(\Omega; \IC^d)}^2
\end{align*}
holds. Hence, the norms of $\H_D^1(\Omega)$ and $\V$ are equivalent on $\C_D^\infty(\Omega)$. As the latter is dense in $\V$ by definition and dense in $\H_D^1(\Omega)$ by Proposition~\ref{Prop: Smooth functions dense in HFs} the conclusion follows. 
\end{proof}

\section{Main Results}
\label{Sec: Main results}

\noindent The purpose of this paper is to solve the Kato Square Root Problem for $A$, i.e.\ to prove the following theorem. Throughout, let $A$ be the elliptic operator defined in Subsection~\ref{Subsec: The Elliptic Operator}.

\begin{maintheorem}
\label{Thm: Kato}
Under Assumptions~\ref{Ass: General geometric assumption on Omega} and \ref{Ass: Ellipticity} the domain of $A^{1/2}$ coincides with the form domain $\V$ and
\begin{align*}
 \|A^{1/2} u\|_{\L^2(\Omega)} \simeq \|\nabla u \|_{\L^2(\Omega; \IC^d)} \qquad (u \in \dom(A^{1/2})).
\end{align*}
\end{maintheorem}

As already outlined in the introduction, we will deduce Theorem~\ref{Thm: Kato} from an extrapolation property of the weak Laplacian with form domain $\V$ defined in Subsection~\ref{Subsec: The Elliptic Operator}. The following theorem accounts for this strategy.

\begin{theorem}[{\cite[Thm.~3.1]{Laplace-Extrapolation-Implies-Kato}}]
\label{Thm: Kato from Laplace}
Let $\Omega \subseteq \IR^d$ be a domain, let $\V$ be a function space that contains $\C_c^\infty(\Omega)$ and that is closed under the norm $\|u\|_{\V}:= (\int_\Omega \abs{u}^2 + \abs{\nabla u}^2)^{1/2}$, and let $\Delta_\V$ be the weak Laplacian defined via the form method with underlying form domain $\V$. Suppose the following hold true.
\begin{enumerate}
\itemsep3pt
 \item[($d$)] The domain $\Omega$ is a $d$-set.
 \item[($d-1$)] The boundary $\bd \Omega$ is a $(d-1)$-set.
 \item[($\V$)] The form domain $\V$ is stable under multiplication by smooth scalar valued functions in the sense that $\varphi \V \subseteq \V$ holds for each $\varphi \in \C_c^\infty(\IR^d)$. Moreover, it has the \emph{$\H^1$ extension property}, i.e.\ there exists a bounded operator $\E_{\V}: \V \to \H^1(\IR^d)$ such that $\E_{\V} u = u$ holds a.e.\ on $\Omega$ for each $u \in \V$.
 \item[($\alpha$)] There exists an $\alpha \in (0, 1)$ such that the complex interpolation space $[\L^2(\Omega), \V]_\alpha$ coincides with $\H^\alpha(\Omega)$ up to equivalent norms.
 \item[(E)] For the \emph{same} $\alpha$ as above $\dom((-\Delta_\V)^{1/2 + \alpha/2}) \subseteq \H^{1+\alpha}(\Omega)$ holds with continuous inclusion.
\end{enumerate}
Then for any $\mu \in \L^\infty(\Omega; \IC^{d \times d})$ that belongs to an elliptic sesquilinear form on $\V$ in the sense of Assumption~\ref{Ass: Ellipticity}, the associated divergence form operator $A = -\nabla \cdot \mu \nabla$ defined via the form method with underlying form domain $\V$ has the \emph{square root property} $\dom(A^{1/2}) = \V$ together with the inhomogeneous estimate
\begin{align*}
 \|(1+A)^{1/2} u\|_{\L^2(\Omega)} \simeq \|u\|_{\L^2(\Omega)} + \|\nabla u \|_{\L^2(\Omega; \IC^d)} \qquad(u \in \dom(A^{1/2})).
\end{align*}
\end{theorem}

Let us outline the major steps in order to deduce our main result from Theorem~\ref{Thm: Kato from Laplace}. First of all Assumption~\ref{Ass: General geometric assumption on Omega} takes care of the demands ($d$) and ($d-1$), see also Remark~\ref{Rem: Remark on geometric assumptions}. The stability property of the particular form domain $\V$ under consideration in this paper is immediate from Definition~\ref{Def: Form domain} and the $\H^1$ extension property has already been discussed in the proof of Corollary~\ref{Cor: Form domain equals HD1}.

Next, the inhomogeneous estimate provided by Theorem~\ref{Thm: Kato from Laplace} already implies the -- at first sight stronger -- homogeneous estimate stated in our main result. The key observation here is that our geometric framework allows for a Poincar\'{e} inequality on $\V \cap \overline{\Rg(A)}$.

\begin{lemma}
\label{Lem: Homogeneous Kato is easy}
Let Assumptions~\ref{Ass: General geometric assumption on Omega} and \ref{Ass: Ellipticity} be satisfied. If $\dom(A^{1/2}) = \V$ holds together with the inhomogeneous estimate $\|(1+A)^{1/2} u\|_{\L^2(\Omega)} \simeq \|u\|_{\L^2(\Omega)} + \|\nabla u \|_{\L^2(\Omega; \IC^d)}$ for all $u \in \dom(A^{1/2})$ then also the homogeneous estimate $\|A^{1/2} u\|_{\L^2(\Omega)} \simeq \|\nabla u \|_{\L^2(\Omega; \IC^d)}$ holds for all such $u$.
\end{lemma}

\begin{proof}
Throughout the proof we abbreviate $\L^2$ norms by $\|\cdot\|_2$. Since $A$ is maximal accretive on the Hilbert space $\L^2(\Omega)$ there is a topological kernel-range splitting $\L^2(\Omega) = \Ke(A) \oplus \cl{\Rg(A)}$, the closure taken in $\L^2(\Omega)$, see e.g.\ \cite[Prop.~2.1.1(h)]{Haase}. For brevity put $Y:= \cl{\Rg(A)}$ and equip it with the $\L^2(\Omega)$ norm. We also need the space $X:= \V \cap Y$ which is closed under the norm $u \mapsto (\int_\Omega \abs{u}^2 + \abs{\nabla u}^2)^{1/2}$ inherited from $\V$. Its meaning stems from the global Poincar\'{e} inequality
\begin{align}
\label{Eq: Poincare}
 \|u\|_2 \lesssim  \|\nabla u\|_2 \qquad (u \in X).
\end{align}
Indeed, within the geometric framework of Assumption~\ref{Ass: General geometric assumption on Omega} and since $\V$ has the $\H^1$ extension property, a sufficient condition for this inequality is that $X$ does not contain non-zero constant functions \cite[Prop.~6.1, Rem.~6.2]{Hardy-Poincare}. But non-zero constant functions that belong to $\V$  also belong to the kernel of $A$ and thus -- by the kernel-range splitting -- cannot be contained in $X$.

Now, define $B$ as the maximal restriction of $A$ to an operator on $Y$. Since the latter space is invariant under resolvents of $A$, the operator $B$ is maximal accretive on the Hilbert space $Y$ and has domain $\dom(B) = \dom(A) \cap Y \subseteq X$. 

The sesquilinear form $\a$ associated to $A$ is elliptic, so $\|Au\|_2 \|u\|_2 \geq \lambda \|\nabla u\|_2^2$ holds for every $u \in \dom(A)$. This implies that firstly $\ker(A)$ only contains constant functions and secondly that due to \eqref{Eq: Poincare} every $w \in \dom(B)$ satisfies the a priori estimate $\|Bw\|_2 \gtrsim \|w\|_2$. Hence, $B$ is injective with closed range and the kernel-range decomposition for maximal accretive operators entails that $B$ is invertible. Note that invertibilty inherits to $B^{1/2}$. This operator is the maximal restriction of $A^{1/2}$ to $Y$ and has domain $\dom(B^{1/2}) = \dom(A^{1/2}) \cap Y$, see \cite[Prop.~2.6.5, Prop.~3.1.1]{Haase} for details. Consequently, $\|A^{1/2} w\|_2 \simeq \|w\|_2 + \|A^{1/2}w\|_2$ for all $w \in \dom(B^{1/2})$. Now, taking into account the assumptions and that $\dom(A^{1/2}) = \dom((1+A)^{1/2})$ holds up to equivalent norms, $\dom(B^{1/2}) = \V \cap Y = X$ follows with equivalences
\begin{align}
\label{Eq: Homogeneous estimate for B}
 \|A^{1/2} w\|_2 \simeq \|w\|_2 + \|\nabla w\|_2 \simeq \|\nabla w\|_2 \qquad (w \in \dom(B^{1/2})).
\end{align}
Here, the second part is due to the Poincar\'{e} inequality \eqref{Eq: Poincare}. 

In order to prove the homogeneous estimate $\|A^{1/2} u\|_2 \simeq \|\nabla u\|_2$ for $u \in \dom(A^{1/2}) = \V$ split $u = v + w$, where ad hoc $v \in \Ke(A)$ and $w \in Y$. We already know $\nabla v = 0$ and $A^{1/2}v = 0$ is immediate from the Balakrishnan Representation \eqref{Balakrishnan Representation}. Moreover, $w$ belongs to $\dom(B^{1/2}) = X = \V \cap Y$ since both $u$ and $v$ belong to $\V$. Hence, \eqref{Eq: Homogeneous estimate for B} applies and the proof is complete.
\end{proof}

Summing up, in order to prove Theorem~\ref{Thm: Kato} it remains to establish ($\alpha$) and the extrapolation property ($E$). In fact we will prove as our main result in this paper the following, considerably stronger statement. Its proof will be developed in the subsequent sections.

\begin{maintheorem}
\label{Thm: Main Result}
Let Assumptions~\ref{Ass: General geometric assumption on Omega} and \ref{Ass: Ellipticity} be satisfied and let $\Delta_\V$ be the weak Laplacian with form domain $\V$. Then 
\begin{align*}
 \dom((-\Delta_\V)^\alpha) = \H^{2\alpha}(\Omega) \qquad (\alpha \in (0,\tfrac{1}{4}))
\end{align*}
and there exists an $\eps \in (0, \frac{1}{4})$ such that 
\begin{align*}
 \dom((-\Delta_\V)^\alpha) = \H_D^{2 \alpha}(\Omega) \qquad (\alpha \in (\tfrac{1}{4}, \tfrac{1}{2} + \eps)).
\end{align*}
In particular, ($\alpha$) and (E) hold for each $\alpha \in (0, 2\eps)$.
\end{maintheorem}

\section{Extension Operators}
\label{Sec: Extension operators}

\noindent The following extension theorem is the main result of this section and at the heart of the interpolation theory for the spaces $\H_D^s(\Omega)$ built up in Section~\ref{Sec: Interpolation for HDs}. An operator $\L^2(\Omega) \to \L^2(\IR^d)$ is called \emph{bounded extension operator}, if it is a bounded right inverse for the restriction operator $\L^2(\IR^d) \to \L^2(\Omega)$.

\begin{theorem}
\label{Thm: Extension operator for HDs}
There exist bounded extension operators $\E, \E_\star: \L^2(\Omega) \to \L^2(\IR^d)$ with the following properties.
\begin{enumerate}
\itemsep3pt
 \item The operator $\E$ restricts to a bounded operator $\H^s(\Omega) \to \H^s(\IR^d)$ if $s \in (0,\frac{1}{2})$ and to a bounded operator $\H_D^s(\Omega) \to \H_D^s(\IR^d)$ if $s \in (\frac{1}{2},\frac{3}{2})$.
 \item The operator $\E_\star$ restricts to a bounded operator $\H^s(\Omega) \to \H^s(\IR^d)$ if $s \in (0,\frac{1}{2})$ and to a bounded operator $\H_D^s(\Omega) \to \H_D^s(\IR^d)$ if $s \in (\frac{1}{2},1)$.
 \item If $f \in \C_D^\infty(\Omega)$ then $\E f$ and $\E_\star f$ have continuous representatives that vanish on $D$.
 \item There is a bounded domain $\Omega_\star \subseteq \IR^d$ that contains $\Omega$ and avoids $D$ such that if $f \in \L^2(\Omega)$ vanishes a.e.\ on a neighborhood of $D$, then $\supp(\E_\star f) \subseteq \Omega_\star$.
\end{enumerate}
\end{theorem}

\begin{remark}
\label{Rem: Remark on reflection extension operator}
The advantage of $\E_\star$ over $\E$ is that for the former we have control on the support of the extended functions. The full meaning of the domain $\Omega_\star$ will become clear in Section~\ref{Sec: Fractional Hardy for HDs}.
\end{remark}

\begin{corollary}
\label{Cor: HDs are reflexive}
The spaces $\H_D^s(\Omega)$, $\frac{1}{2} < s < \frac{3}{2}$, and $\H^s(\Omega)$, $0 \leq s < \frac{1}{2}$, are reflexive.
\end{corollary}

\begin{proof}
Let $\frac{1}{2} < s < \frac{3}{2}$. First, $\H_D^s(\IR^d)$ is reflexive as a closed subspace of the reflexive space $\H^s(\IR^d)$. Since $\E: \H_D^s(\Omega) \to \H_D^s(\IR^d)$ is a bounded right-inverse for the restriction operator $\R: \H_D^s(\IR^d) \to \H_D^s(\Omega)$, it immediately follows that $\E$ is an isomorphism from $\H_D^s(\Omega)$ onto the closed subspace $\E(\H_D^s(\Omega))$ of $\H_D^s(\IR^d)$. The argument in the case $0 \leq s < \frac{1}{2}$ is similar.
\end{proof}

We will prove Theorem~\ref{Thm: Extension operator for HDs} in Subsection~\ref{Subsec: Proof of universal extension theorem} below. Following \cite{Rehberg-Jonsson}, the underlying strategy is:
\begin{center}
\textit{Extend by zero over $D$ and use bi-Lipschitz charts to extend over $\cl{\bd \Omega \setminus D}$.}
\end{center}
This suggests to study the zero extension operator
\begin{align*}
 \E_0: \L^2(\Omega) \to \L^2(\IR^d), \quad (\E_0 f)(x) = \begin{cases}
                                                         f(x), & \text{if $x\in \Omega$,} \\ 0, & \text{if $x \in \IR^d \setminus \Omega$,}
                                                        \end{cases}
\end{align*}
first. Recall from Remark~\ref{Rem: Remark on geometric assumptions} that $\bd \Omega$ is a $(d-1)$-set. While obviously $\E_0$ is bounded from $\L^2(\Omega)$ into $\L^2(\IR^d)$ as well as from $\H_{\bd \Omega}^1(\Omega)$ into $\H_{\bd \Omega}^1(\IR^d)$ (for the latter use that $\C_{\bd \Omega}^\infty(\Omega)$ is dense in $\H_{\bd \Omega}^1(\Omega)$, cf.\ Proposition~\ref{Prop: Smooth functions dense in HFs}) the question whether it acts boundedly between fractional Sobolev spaces is much more involved. Roughly speaking, the problem stems from the non-local norm of these spaces. Our main result on zero extensions is the following.

\begin{theorem}
\label{Thm: Zero extension bounded}
The operator $\E_0$ restricts to a bounded operator $\H^s(\Omega) \to \H^s(\IR^d)$ if $s \in [0, \frac{1}{2})$ and to a bounded operator $\H_{\bd \Omega}^s(\Omega) \to \H_{\bd \Omega}^s(\IR^d)$ if $s \in (\frac{1}{2},\frac{3}{2})$.
\end{theorem}

The proof of Theorem~\ref{Thm: Zero extension bounded} is presented in the next subsection. For a clear presentation of the proofs we introduce the following notion.

\begin{definition}
\label{Def: Hs boundedness}
Let $\Xi_1, \Xi_2 \subseteq \IR^d$ be domains and $s\geq 0$. An operator $T: \L^2(\Xi_1) \to \L^2(\Xi_2)$ is called \emph{$\H^s$ bounded} if it restricts to a bounded operator from $\H^s(\Xi_1)$ into $\H^s(\Xi_2)$.
\end{definition}

\subsection{The proof of Theorem~\ref{Thm: Zero extension bounded}}
\label{Subsec: Boundedness of zero extension}

The proof of Theorem~\ref{Thm: Zero extension bounded} is divided into two consecutive steps.

\subsection*{Step 1: Fractional Hardy inequalities}

The strategy of proof is to use an intrinsic connection between $\H^s$ boundedness of $\E_0$ and the fractional Hardy inequality. This idea is taken from \cite{Hardy-Triebel-LizorkinII}.

\begin{lemma}
\label{Lem: Reduction to Hardy estimate}
For each $s \in (0,1)$ the zero extension operator $\E_0$ satisfies
\begin{align*}
 [\E_0 f]_s^2 \lesssim \iint_{\substack{x,y \in \Omega \\ \abs{x-y} < 1}} \frac{\abs{f(x)-f(y)}^2}{\abs{x-y}^{d+2s}} \; \d x \; \d y +  \int_\Omega \frac{\abs{f(x)}^2}{\dist_{\bd \Omega}(x)^{2s}} \; \d x \qquad (f \in \H^s(\Omega)).
\end{align*}
\end{lemma}

\begin{proof}
Set $M:= \{(x,y) \in \IR^d \times \IR^d: \abs{x-y}<1\}$ and note that if $s \in (0,1)$ and $f \in \H^s(\Omega)$ then
\begin{align*}
 [\E_0 f]_s^2= \int_\Omega \int_\Omega \frac{\abs{f(x)-f(y)}^2}{\abs{x-y}^{d+2s}} \ind_M(x,y) \; \d x \; \d y + 2\int_\Omega \abs{f(y)}^2 \int_{\IR^d \setminus \Omega} \frac{1}{\abs{x-y}^{d+2s}} \ind_M(x,y)\; \d x \; \d y.
\end{align*}
Since for each $y \in \Omega$ the ball $B(y, \dist_{\bd \Omega}(y))$ is contained in $\Omega$, the desired estimate follows from
\begin{align*}
 \int_{\IR^d \setminus \Omega} \frac{1}{\abs{x-y}^{d+2s}} \ind_M(x,y)\; \d x \leq \int_{\IR^d \setminus B(y, \dist_{\bd \Omega}(y))} \frac{1}{\abs{x-y}^{d+2s}} \; \d x \simeq \frac{1}{\dist_{\bd \Omega}(y)^{2s}} \qquad(y \in \Omega). &\qedhere
\end{align*}
\end{proof}

Up to technical details, Lemma~\ref{Lem: Reduction to Hardy estimate} reduces the claim of Theorem~\ref{Thm: Zero extension bounded} to the question whether the $\L^2(\Omega)$ norm of $\abs{f}\dist_{\bd \Omega}^{-s}$ can be controlled in terms of $\|f\|_{\H^s(\Omega)}$ or $\|f\|_{\H_{\bd \Omega}^s(\Omega)}$, respectively. Such an estimate is called a \emph{fractional Hardy inequality}. The subsequent propositions summarize the state of the art concerning such inequalities in our geometric setting.

\begin{proposition}
\label{Prop: Hardy for vanishing boundary trace I}
If $s \in (0,\frac{1}{2})$ then the following fractional Hardy inequality holds true:
\begin{align*}
 \int_\Omega \frac{\abs{f(x)}^2}{\dist_{\bd \Omega}(x)^{2s}} \; \d x \lesssim \|f\|_{\H^s(\Omega)}^2 \qquad(f \in \H^s(\Omega)).
\end{align*}
\end{proposition}

The proof of Proposition~\ref{Prop: Hardy for vanishing boundary trace I} is given in \cite[Thm.\ 1.2]{Hardy-Triebel-LizorkinII} under a weaker geometric assumption on $\Omega$ as in the present paper. The reader may consult \cite[Lem.\ 2.1]{Lehrback-Boundarysize} for a proof that the condition on the \emph{Aikawa dimension} of $\bd \Omega$ in \cite{Hardy-Triebel-LizorkinII} is indeed weaker than that of $\bd \Omega \subseteq \IR^d$ being a $(d-1)$-set.

In the case $s \in (\frac{1}{2},1)$ we can rely on Theorem 2 and Proposition 8 in \cite{Dyda-Vahakangas} where the fractional Hardy inequality occurring in Proposition~\ref{Prop: Hardy for vanishing boundary trace II} is proved for $f \in \C_{\bd \Omega}^\infty(\Omega)$ under the present assumptions on $\Omega$, i.e.\ that it is bounded, plump, and that its boundary is a $(d-1)$-set, cf.\ Remark~\ref{Rem: Remark on geometric assumptions}. To be precise, the reader should invoke the easy part of Frostman's lemma \cite[Thm.~5.1.12]{Adams-Hedberg} to check that the \emph{fatness condition} in \cite{Dyda-Vahakangas} is again satisfied if $\bd \Omega$ is a $(d-1)$-set. Taking into account Proposition~\ref{Prop: Smooth functions dense in HFs} we can record the following result.

\begin{proposition}
\label{Prop: Hardy for vanishing boundary trace II}
If $s \in (\frac{1}{2},1)$ then the following fractional Hardy inequality holds true:
\begin{align*}
 \int_\Omega \frac{\abs{f(x)}^2}{\dist_{\bd \Omega}(x)^{2s}} \; \d x \lesssim \int_\Omega \int_\Omega \frac{\abs{f(x) - f(y)}^2}{\abs{x-y}^{d+ 2s}} \; \d x \; \d y \lesssim \|f\|_{\H_{\bd \Omega}^s(\Omega)}^2 \qquad(f \in \H_{\bd \Omega}^s(\Omega)). 
\end{align*}
\end{proposition}

\subsection*{Step 2: $\H^s$ boundedness of $\E_0$}

The cases $s=0$ and $s = 1$ have already been discussed. If $s \in (0,\frac{1}{2})$ then Lemma~\ref{Lem: Reduction to Hardy estimate} and Proposition~\ref{Prop: Hardy for vanishing boundary trace I} yield $[\E_0 f]_s^2 \lesssim \|f\|_{\H^s(\Omega)}^2$ for each $f \in \H^s(\Omega)$ and since $\E_0$ is $\L^2$ bounded the conclusion follows. Likewise, if $s \in (\frac{1}{2},1)$ it follows from Lemma~\ref{Lem: Reduction to Hardy estimate} and Proposition~\ref{Prop: Hardy for vanishing boundary trace II} that $\E_0$ maps $\H_{\bd \Omega}^s(\Omega)$ boundedly into $\H^s(\IR^d)$ and it remains to check that in fact $\E_0 f \in \H_{\bd \Omega}^s(\IR^d)$ if $f \in \H_{\bd \Omega}^s(\Omega)$. This is certainly true if $f \in \C_{\bd \Omega}^\infty(\Omega)$ and thus follows for general $f \in \H_{\bd \Omega}^s(\Omega)$ by approximation, cf.\ Proposition~\ref{Prop: Smooth functions dense in HFs}. 

Finally, let $s \in (1,\frac{3}{2})$ and $f \in \H_{\bd \Omega}^s(\Omega) \subseteq  \H_{\bd \Omega}^1(\Omega)$. The assertion for $s=1$ yields
\begin{align*}
 \|\E_0 f\|_{\H^s(\IR^d)} = \|\E_0 f\|_{\H^1(\IR^d)} + \sum_{j=1}^d [\partial_j (\E_0f) ]_{s-1} \lesssim \|f\|_{\H_{\bd \Omega}^s(\Omega)} + \sum_{j=1}^d [\partial_j (\E_0f) ]_{s-1}.
\end{align*}
Note $\partial_j (\E_0 f) = \E_0 (\partial_j f)$ for $1\leq j \leq d$, as is obvious if $f \in \C_{\bd \Omega}^\infty(\Omega)$ and then extends to general $f \in \H_{\bd \Omega}^1(\Omega)$ by density. Since the derivation operators $\partial_j$ are bounded from $\H_{\bd \Omega}^s(\Omega)$ into $\H^{s-1}(\Omega)$, the assertion for $s-1$ yields
\begin{align*}
 [\partial_j (\E_0f) ]_{s-1} = [\E_0 (\partial_j f) ]_{s-1} \leq \|\E_0(\partial_j f)\|_{\H^{s-1}(\IR^d)} \lesssim \|\partial_j f\|_{\H^{s-1}(\Omega)} \lesssim \|f\|_{\H_{\bd \Omega}^s(\Omega)} \qquad (1 \leq j \leq d).
\end{align*}
Altogether, $\|\E_0 f\|_{\H^s(\IR^d)} \lesssim \|f\|_{\H_{\bd \Omega}^s(\Omega)}$. To conclude, note that $\E_0 f \in \E_0(\H_{\bd \Omega}^1(\Omega)) \subseteq \H_{\bd \Omega}^1(\IR^d)$ implies $\R_{\bd \Omega}(\E_0 f) = 0$, so that in fact $\E_0 f$ is in $\H_{\bd \Omega}^s(\IR^d)$. \hfill $\square$

\subsection{The proof of Theorem~\ref{Thm: Extension operator for HDs}}
\label{Subsec: Proof of universal extension theorem}

The argument is divided into six consecutive steps.

\subsection*{Step 1: Local extension operators}

Since  $\cl{\bd \Omega \setminus D}$ is compact we can, according to Assumption~\ref{Ass: General geometric assumption on Omega}, fix an open covering $\bigcup_{j=1}^n U_j$ of $\cl{\bd \Omega \setminus D}$ with the following property: For $1 \leq j \leq n$ there is a bi-Lipschitz map $\Phi_j$ from $U_j$ onto the open unit cube $(-1,1)^d$ such that 
\begin{align*}
 \Phi_j(\Omega_j) = (-1,1)^{d-1} \times (-1,0) \quad \text{and} \quad \Phi_j(\bd \Omega \cap U_j) = (-1,1)^{d-1} \times \{0\},
\end{align*}
where $\Omega_j: = \Omega \cap U_j$. We can assume that none of the sets $U_j$ is superfluous i.e.\ that $\cl{\bd \Omega \setminus D} \cap U_j \neq \emptyset$ for all $j$. With this convention $n=0$ in the case $D = \bd \Omega$.

To proceed further, we recall the following deep result of Rogers \cite[Thm.\ 8]{Rogers-Article}.

\begin{theorem}[Rogers]
\label{Thm: Rogers extension}
Let $\Xi \subseteq \IR^d$ be a domain for which there are constants $\eps, \delta > 0$ such that between each pair of points $x,y \in \Xi$ with $\abs{x-y} < \delta$ there is a rectifiable arc $\gamma \subseteq \Xi$ of length at most $\eps^{-1} \abs{x-y}$  having the property
\begin{align*}
 \d_{\bd \Xi}(z) \geq \frac{\eps \abs{x-z} \abs{y-z}}{\abs{x-y}} \qquad (z \in \gamma).
\end{align*}
Then there exists a bounded extension operator $\E: \L^2(\Xi) \to \L^2(\IR^d)$ that restricts to a bounded operator $\H^k(\Xi) \to \H^k(\IR^d)$ for each $k \in \IN$.
\end{theorem}

\begin{remark}
\label{Rem: Rogers result I}
\begin{enumerate}
\itemsep3pt
 \item In fact Rogers' extension operator is also bounded on Sobolev spaces in the $\L^p$ scale for each $p \in [1,\infty]$. To avoid confusion let us remark that all results in \cite{Rogers-Article} are formulated for Sobolev spaces only, but throughout the $\L^p$ case $k=0$ is allowed.
 
 \item A domain satisfying the quantitative connectedness condition of Theorem~\ref{Thm: Rogers extension} is usually called \emph{$(\eps, \delta)$-domain} or \emph{locally uniform domain}. For further information and a comparison to related geometric concepts we refer to \cite{Rogers-Article, Vaisala} and references therein.

\end{enumerate}
\end{remark}

\begin{remark}
\label{Rem: Rogers result II}
The premise of Theorem~\ref{Thm: Rogers extension} is in particular satisfied for $\Xi = (-1,1)^{d-1} \times (-1,0)$: Indeed, it is straightforward -- but admittedly a little tedious -- to check that in this case for each pair $x,y \in \Xi$ the arc $\gamma$ can be constructed by first choosing a sub cube $Q_{x,y} \subseteq \Xi$ with side length $\frac{1}{\sqrt{2}} \abs{x-y}$ and then connecting both $x$ and $y$ with the center of $Q_{x,y}$ by straight lines.

Moreover, if $\Xi$ satisfies the premise of Theorem~\ref{Thm: Rogers extension} then so does every bi-Lipschitz image of it. As connecting arcs in the image of $\Xi$ simply take the images of the connecting arcs in $\Xi$. In particular, Theorem~\ref{Thm: Rogers extension} applies to $\Xi = \Omega_j$ for $1 \leq j \leq n$. A refinement of this argument yields the well-known fact that every bounded Lipschitz domain is an $(\eps,\delta)$-domain, cf.\ \cite[Ch.\ 3]{Triebel-Wavelets}. 
\end{remark}

If only a bounded extension operator for first order Sobolev spaces is needed, we can rely on an easy reflection technique instead: 

\medskip

\begin{center}
\begin{minipage}{0.55\textwidth}
 \textit{Transform $\Omega_j$ to the lower half-cube, extend to the unit cube by even reflection and transform back to $U_j$.}
\end{minipage}
\end{center}

\smallskip

\noindent This has the advantage of a control on the extended function outside of $\Omega$ needed later on for the construction of $\E_\star$. More precisely we have the following lemma whose easy proof is omitted.

\begin{lemma}
\label{Lem: Reflection extension}
Let $1 \leq j \leq n$ and denote by 
\begin{align*}
 \mathfrak{S}: \L^2((-1,1)^{d-1} \times (-1,0)) \to \L^2((-1,1)^d), \quad (\mathfrak{S} f)(x) = f(x_1,\ldots, x_{d-1}, -\sgn(x_d)x_d)
\end{align*} 
the extension operator by even reflection. Then
\begin{align*}
 \E_{\star,j}: \L^2(\Omega_j) \to \L^2(U_j), \quad (\E_{\star,j} f)(x) = \mathfrak{S}(f\circ \Phi_j^{-1})(\Phi_j(x))
\end{align*}
is a bounded extension operator that maps $\H^1(\Omega_j)$ boundedly into $\H^1(U_j)$.
\end{lemma}

\subsection*{Step 2: Construction and $\H^s$ boundedness of $\E$.}

First, fix bounded extension operators $\E_j: \L^2(\Omega_j) \to \L^2(\IR^d)$, $1 \leq j \leq n$, according to Theorem~\ref{Thm: Rogers extension}. Also fix a cut-off function $\eta \in \C_c^\infty(\IR^d)$ that is identically one in a neighborhood of $\cl{\bd \Omega \setminus D}$ and has its support in $\bigcup_{j=1}^n U_j$. Let $\eta_1, \ldots, \eta_n$ be a smooth partition of unity on $\supp(\eta)$ subordinated to $U_1, \ldots, U_n$. Finally, take cut-off functions $\chi_j \in \C_c^\infty(U_j)$, $1 \leq j \leq n$, with $\chi_j$ identically one on $\supp(\eta_j)$. With this notation put
\begin{align}
\label{Definition of universal extension operator}
 \E: \L^2(\Omega) \to \L^2(\IR^d), \quad \E f = \E_0 ((1-\eta)f) + \sum_{j = 1}^n \chi_j \E_j(\eta_j \eta f),
\end{align}
where $\E_0$ is the zero extension operator introduced at the beginning of Section~\ref{Sec: Extension operators}. Note that $\E$ is indeed an extension operator since for $f \in \L^2(\Omega)$ the restriction of $\E f$ to $\Omega$ coincides with
\begin{align*}
 (1-\eta)f + \sum_{j = 1}^n \chi_j \eta_j \eta f = (1-\eta)f + \sum_{j = 1}^n \eta_j \eta f = (1-\eta)f + \eta f = f.
\end{align*}
In the remainder of this step we prove that $\E$ restricts to a bounded operator $\H_D^s(\Omega) \to \H^s(\IR^d)$ if $s \in (\frac{1}{2},\frac{3}{2})$. That $\E$ in fact maps  $\H_D^s(\Omega)$ into $\H_D^s(\IR^d)$ is postponed until Step 5. Upon replacing the symbol $\H_F^s$ by $\H^s$ for any $(d-1)$-set $F$ occurring in the following, literally the same argument shows that $\E$ restricts to a bounded operator $\H^s(\Omega) \to \H^s(\IR^d)$ if $s \in [0,\frac{1}{2})$. 

For the rest of the proof fix $f \in \H_D^s(\Omega)$. Throughout, implicit constants may depend on all other parameters but on $f$.

Since $1-\eta$ vanishes on $\bd \Omega \setminus D$, the multiplication operator associated to $1-\eta$ maps $\H_D^s(\Omega)$ boundedly into $\H_{\bd \Omega}^s(\Omega)$, cf.\ Lemma~\ref{Lem: Multiplication trace lemma}. Invoking Theorem~\ref{Thm: Zero extension bounded}, we find
\begin{align}
\label{First estimate universal extension operator}
\|\E_0 ((1-\eta)f)\|_{\H^s(\IR^d)} \lesssim \|(1-\eta)f \|_{\H_{\bd \Omega}^s(\Omega)} \lesssim \|f\|_{\H_D^s(\Omega)}.
\end{align}
Concerning the remaining terms in \eqref{Definition of universal extension operator} note that for $1 \leq j \leq n$ Lemma~\ref{Lem: Multiplication trace lemma} yields
\begin{align*}
 \|\eta_j \eta f\|_{\H^s(\Omega_j)} \leq \|\eta_j \eta f\|_{\H^s(\Omega)} \lesssim \|f\|_{\H^s(\Omega)} \leq \|f\|_{\H_D^s(\Omega)}
\end{align*}
and
\begin{align*}
 \|\chi_j\E_j(\eta_j \eta f)\|_{\H^s(\IR^d)} \lesssim \|\E_j(\eta_j \eta f)\|_{\H^s(\IR^d)}
\end{align*}
since $\eta_j \eta$ and $\chi_j$ are smooth and compactly supported. Hence, the only task is to prove $\H^s$ boundedness of $\E_j$. By construction $\E_j$ is $\H^k$ bounded if $k = 0,2$. Since the restriction operators $\H^k(\IR^d) \to \H^k(\Omega_j)$ are bounded, the retraction-coretraction theorem \cite[Sec.\ 1.2.4]{Triebel} together with the complex interpolation result $[\L^2(\IR^d), \H^2(\IR^d)]_{s/2}~=~\H^s(\IR^d)$, see e.g.\ \cite[Thm.\ 6.4.5]{Bergh-Loefstroem}, yields that $\E_j(\H^s(\Omega_j))$ is a closed subspace of $\H^s(\IR^d)$ and that
\begin{align}
\label{Third estimate universal extension operator}
 \E_j: \big[\L^2(\Omega_j), \H^2(\Omega_j)\big]_{s/2} \to \E_j(\H^s(\Omega_j)) 
\end{align}
is an isomorphism. Hence, $\H^s(\Omega_j)$ and $[\L^2(\Omega_j), \H^2(\Omega_j)]_{s/2}$ coincide as sets and due to
\begin{align*}
 \|h\|_{\H^s(\Omega_j)} \leq \|\E_j h \|_{\H^s(\IR^d)} \lesssim \|h\|_{[\L^2(\Omega_j), \H^2(\Omega_j)]_{s/2}} \qquad (h \in \H^s(\Omega_j))
\end{align*}
and the bounded inverse theorem they also coincide as Banach spaces. Now, \eqref{Third estimate universal extension operator} yields $\H^s$ boundedness of $\E_j$ and the boundedness of $\E: \H_D^s(\Omega) \to \H^s(\IR^d)$ follows.

\subsection*{Step 3: Construction and $\H^s$ boundedness of $\E_\star$.}

For the construction of $\E_\star$ we rely on the same pattern as for $\E$ but use $\E_{\star,j}$, $1 \leq j \leq n$, defined in Lemma~\ref{Lem: Reflection extension} as local extension operators. Since these are only extension operators from $\L^2(\Omega_j)$ into $\L^2(U_j)$, we introduce the respective zero extension operators $\E_{0,j}: \L^2(U_j) \to \L^2(\IR^d)$. With $\eta$, $\eta_j$, and $\chi_j$ as in Step 2 we then put
\begin{align}
\label{Definition reflection extension operator}
 \E_\star: \L^2(\Omega) \to \L^2(\IR^d), \quad \E_\star f = \E_0 ((1-\eta)f) + \sum_{j = 1}^n \E_{0,j}(\chi_j \E_{\star,j}(\eta_j \eta f)).
\end{align}
In analogy with Step 2 we focus on $s \in (\frac{1}{2},1)$ and prove that $\E_\star$ restricts to a bounded operator $\H_D^s(\Omega) \to \H^s(\IR^d)$. The zero extension term in \eqref{Definition reflection extension operator} has already been taken care of in \eqref{First estimate universal extension operator} so that it suffices to consider the terms containing $\E_{\star,j}$. 

For $k=0,1$ Lemmas \ref{Lem: Multiplication trace lemma} and \ref{Lem: Reflection extension} yield that $M_{\chi_j} \E_{\star,j} M_{\eta_j \eta}$ is bounded from $\H^k(\Omega_j)$ into $\H^k(U_j)$. Here, as usual, $M$ denotes the corresponding multiplication operator. Since $\chi_j$ has compact support in $U_j$ it follows that $\E_{0,j}M_{\chi_j} \E_{\star,j} M_{\eta_j \eta}$ maps $\H^k(\Omega_j)$ boundedly into $\H^k(\IR^d)$. Due to $s<1$ the same interpolation argument as in Step 2 reveals $[\L^2(\Omega_j), \H^1(\Omega_j)]_s = \H^s(\Omega_j)$ if one relies on the $\H^1$ boundedness of $\E_j$ rather than on its $\H^2$ boundedness. Hence, by complex interpolation, $\E_{0,j} M_{\chi_j} \E_{\star,j} M_{\eta_j \eta}$ maps $\H^s(\Omega_j)$ boundedly into $\H^s(\IR^d)$, i.e.\ 
\begin{align*}
 \|\E_{0,j}(\chi_j \E_{\star,j}(\eta_j \eta f))\|_{\H^s(\IR^d)} \lesssim \|f\|_{\H^s(\Omega_j)} \leq \|f\|_{\H_D^s(\Omega)} \qquad (f \in \H_D^s(\Omega)).
\end{align*}
Going back to \eqref{Definition reflection extension operator} the boundedness of $\E_\star: \H_D^s(\Omega) \to \H^s(\IR^d)$ follows.

\subsection*{Step 4: $\E$ and $\E_\star$ map test functions to continuous functions that vanish on $D$}

The purpose of this step is to prove the third part of Theorem~\ref{Thm: Extension operator for HDs}. To this end, let $f \in \C_D^\infty(\Omega)$. Recall from \eqref{Definition of universal extension operator} that $\E f$ is given by
\begin{align*}
 \E f = \E_0 ((1-\eta)f) + \sum_{j = 1}^n \chi_j \E_j(\eta_j \eta f),
\end{align*}
where $\eta$ is smooth and identically one in a neighborhood of $\cl{\bd \Omega \setminus D}$, the functions $\chi_j$ and $\eta_j$ are smooth, and the local extension operators $\E_j: \L^2(\Omega_j) \to \L^2(\IR^d)$ are chosen according to Theorem~\ref{Thm: Rogers extension}. Due to $(1-\eta) f \in \C_c^\infty(\Omega)$ its zero extension $\E_0 ((1-\eta)f)$ is smooth on $\IR^d$. For $1 \leq j \leq n$ note that $\eta_j \eta f$ is in $\H^k(U_j)$ for each $k \in \IN$ and hence that $\chi_j \E_j(\eta_j \eta f)$ is in $\H^k(\IR^d)$ for each $k \in \IN$ thanks to Theorem~\ref{Thm: Rogers extension}. Choosing $k$ large enough it follows by Sobolev embeddings that $\chi_j \E_j(\eta_j \eta f)$ has a continuous representative; and thus so has $\E f$.

To prove that $\E_\star f$ has a continuous representative is even easier. Instead of Sobolev embeddings simply use that even reflection from the lower half to the full unit cube preserves continuity.

Finally, let $\mathfrak{f}$ be the continuous representative for $\E f$ and $\E_\star f$, respectively. By assumption there is an open set $U \supseteq D$ such that $f = 0$ a.e.\ on $U \cap \Omega$. Thus, $\mathfrak{f}$ vanishes on $U \cap \Omega$. Since every point $x \in D$ is an accumulation point of $U \cap \Omega$ it follows by continuity that $\mathfrak{f}$ vanishes on $D$.

\subsection*{Step 5: $\E$ and $\E_\star$ map into spaces with vanishing trace on $D$}

To conclude the proof of the first two items of Theorem~\ref{Thm: Extension operator for HDs} we have yet to show that $\E$ and $\E_\star$ in fact map $\H_D^s(\Omega)$ into $\H_D^s(\IR^d)$ if $s \in (\frac{1}{2}, \frac{3}{2})$ and $s \in (\frac{1}{2},1)$, respectively. Since the proofs are almost the same we concentrate on $\E$. Also, only the case $D \neq \emptyset$ is of interest.

Let $s \in (\frac{1}{2}, \frac{3}{2})$, $f \in \H_D^s(\Omega)$, and pick some $t \in (\frac{1}{2},1)$ not larger than $s$. Use Proposition~\ref{Prop: Smooth functions dense in HFs} to approximate $f$ in $\H_D^t(\Omega)$ by a sequence $(f_n)_n \subseteq \C_D^\infty(\Omega)$. Step~2 infers that $(\E f_n)_n$ converges to $\E f$ in $\H^t(\IR^d)$. Thanks to Step 4 each $\E f_n$ has a continuous representative that vanishes on $D$. Hence, $\R_D \E f_n = 0$ for each $n \in \IN$ and therefore $\R_D \E f =0$ by continuity of $\R_D$, see Proposition~\ref{Prop: Jonsson-Wallin trace theorem}. But this exactly means that $\E f$ not only belongs to $\H^s(\IR^d)$ but to $\H_D^s(\IR^d)$.

\subsection*{Step 6: The support property of $\E_\star$.}

To prove the last item of Theorem~\ref{Thm: Extension operator for HDs} let $f \in \L^2(\Omega)$ be such that there is an open set $U \supseteq D$ with $f=0$ a.e.\ on $\Omega \cap U$. Then $(1-\eta)f$ has compact support in $\Omega$ and clearly so has $\E_0(1-\eta)f$. If $1 \leq j \leq n$ then $\eta \eta_j$ has compact support in $U_j$. Hence, $\E_{\star,j}(\eta \eta_j f)$ has compact support in $U_j \setminus D$ by construction of $\E_{\star,j}$, see Lemma~\ref{Lem: Reflection extension}, and the same remains true for $\E_{0,j}(\chi_j \E_{\star,j}(\eta \eta_j f))$. In a nutshell, $\E_\star f$ has compact support in
\begin{align*}
  \Omega_\star: = \Omega \cup \bigcup_{j=1}^n (U_j \setminus D),
\end{align*}
see \eqref{Definition reflection extension operator}.
Clearly $\Omega_\star$ is open, contains $\Omega$ and avoids $D$. The sets $U_j \setminus D$ are contained in bi-Lipschitz images of the open unit cube and therefore are bounded. Hence, $\Omega_\star$ is bounded and it remains to show that it is connected. Since the union of connected sets with a common point is again connected, it suffices to show that for $1 \leq j \leq n$ the set $U_j \setminus D$ is connected and has non-empty intersection with $\Omega$.

By construction $U_j$ intersects $\cl{\bd \Omega \setminus D}$. Since $U_j$ is open it must intersect both $\Omega$ and $\bd \Omega \setminus D$. The latter implies
that $\Phi_j(U_j \setminus D) \subseteq (-1,1)^d$ does not only contain the lower and upper open half of the unit cube but also a point from their common frontier $(-1,1)^{d-1} \times \{0\}$. From this it follows that $\Phi_j(U_j \setminus D)$ is (arcwise) connected and by continuity of $\Phi_j^{-1}$ the same holds for $U_j \setminus D$. This completes the proof of Theorem~\ref{Thm: Extension operator for HDs}. \hfill $\square$

\section{A Fractional Hardy Type Inequality}
\label{Sec: Fractional Hardy for HDs}

\noindent The result we want to prove in this section is the following fractional Hardy type inequality for functions that, in contrast to the inequalities presented in Subsection~\ref{Subsec: Boundedness of zero extension}, only vanish on the Dirichlet part $D$ of the boundary of $\Omega$.

\begin{theorem}
\label{Thm: Fractional Hardy for HDs}
If $s \in (\frac{1}{2},1)$ then the following fractional Hardy type inequality holds true:
\begin{align}
\label{Eq: Fractional Hardy for HDs}
 \int_\Omega \frac{\abs{f(x)}^2}{\dist_{D}(x)^{2s}} \; \d x \lesssim \|f\|_{\H_D^s(\Omega)}^2 \qquad (f \in \H_D^s(\Omega)).
\end{align}
\end{theorem}

Since the statement of Theorem~\ref{Thm: Fractional Hardy for HDs} is void if $D = \emptyset$, we exclude this case for the entire section. The proof of Theorem~\ref{Thm: Fractional Hardy for HDs} extends the ideas of \cite[Sec.\ 6]{ABHR}, where a Hardy type inequality for first order Sobolev spaces with partially vanishing boundary traces was shown. 

The following concept of \emph{fat sets} turned out to be essential in the area of Hardy inequalities, see e.g.\ \cite{Lewis-FatSets}, \cite{Lehrback-PointwiseHardy}, \cite{Boundary-Visibility}. First, the \emph{Riesz kernels} of order $s>0$ on $\IR^d$ are given by $I_s(x):= \abs{x}^{s-d}$. If $0< 2s < d$ define the \emph{$(s,2)$-outer capacity} of subsets $E \subseteq \IR^d$ by
\begin{align*}
 R_{s,2}(E):= \inf \left\{ \|f\|_{\L^2(\IR^d)}^2: \text{$f \geq 0$ on $\IR^d$ and $f \ast I_s \geq 1$ on $E$} \right\}.
\end{align*}
A set $E \subseteq \IR^d$ is then called \emph{$(s,2)$-uniformly fat} if
\begin{align*}
 R_{s,2}(E \cap B(x,r)) \gtrsim r^{d-2s} \qquad (x \in E, \, r>0).
\end{align*}
Finally, the $(d-1)$-dimensional \emph{Hausdorff content} of $E \subseteq \IR^d$ is
\begin{align*}
 \m_{d-1}^\infty(E):= \inf \Big \{ \sum_{j= 1}^\infty r_j^{d-1}: x_j \in E, \, r_j > 0, \, E \subseteq \bigcup_{j=1}^\infty B(x_j, r_j) \Big \}.
\end{align*}

Next, let us quote the deep results from geometric measure theory that relate $(s,2)$-uniformly fat sets to our geometric setting.

\begin{proposition}[{\cite[Prop.\ 3.11]{Boundary-Visibility}}]
\label{Prop: Thickness implies fatness}
If the complement of a domain $\Xi \subseteq \IR^d$ satisfies the \emph{thickness condition}
\begin{align}
\label{Thickness condition}
 \m_{d-1}^\infty(\Xi^c \cap B(x,r)) \gtrsim r^{d-1} \qquad (x \in \Xi^c,\, r > 0) 
\end{align}
then it is $(s,2)$-uniformly fat for each $1 < 2s <d$.
\end{proposition}

\begin{proposition}[{\cite[pp.\ 2197-2198]{Lehrback-PointwiseHardy}}]
\label{Prop: Boundary density implies thickness}
If a domain $\Xi \subseteq \IR^d$ satisfies the \emph{inner boundary density condition}
\begin{align}
\label{Boundary density condition}
 \m_{d-1}^\infty\big(\bd \Xi \cap B(x, 2 \dist_{\bd \Xi}(x))\big) \gtrsim \dist_{\bd \Xi}(x)^{d-1} \qquad (x \in \Xi)
\end{align}
then its complement satisfies the thickness condition \eqref{Thickness condition}.
\end{proposition}

\begin{lemma}
\label{Lem: d-1 set implies boundary density condition}
Each \emph{bounded} domain $\Xi \subseteq \IR^d$ whose boundary is a $(d-1)$-set satisfies the inner boundary density condition \eqref{Boundary density condition} -- and thus has $(s,2)$-uniformly fat complement for $1<2s<d$.
\end{lemma}

\begin{proof}
Fix $x \in \Xi$, put $E:= \bd \Xi \cap B(x, 2 \dist_{\bd \Xi}(x))$, and let $\{B(x_j,r_j)\}_{j \in \IN}$ be a covering of $E$ by open balls centered in $E$. If $r_j \leq 1$ then $r_j^{d-1}$ is comparable to $\m_{d-1}(\bd \Xi \cap B(x_j,r_j))$ and if $r_j >1$ then certainly $\m_{d-1}(\bd \Xi \cap B(x_j,r_j)) \leq \m_{d-1}(\bd \Xi) r_j^{d-1}$. Note that $0 < \m_{d-1}(\bd \Xi) < \infty$ holds since by boundedness of $\Xi$ one can cover $\bd \Xi$ by finitely many balls of radius $1$ centered in $\bd \Xi$. Thus,
\begin{align*}
 \sum_{j= 1}^\infty r_j^{d-1} \gtrsim \sum_{j= 1}^\infty \m_{d-1}(\bd \Xi \cap B(x_j, r_j)) \geq \m_{d-1}\Big(\bd \Xi \cap \bigcup_{j=1}^\infty B(x_j,r_j) \Big) \geq \m_{d-1}(E).
\end{align*}
On the other hand, if $y \in \bd \Xi$ realizes $\dist_{\bd \Xi}(x)$ then $B(y, \dist_{\bd \Xi}(x)) \subseteq B(x, 2 \dist_{\bd \Xi}(x))$ so that item (1) of Remark~\ref{Rem: Remarks on d-1 set} applied with $r_0 = \diam(\Xi)$ yields
\begin{align*}
 \m_{d-1}(E) \geq \m_{d-1}\big(\bd \Xi \cap B(y, \dist_{\bd \Xi}(x))\big) \gtrsim \dist_{\bd \Xi}(x)^{d-1}.
\end{align*}
Now, the conclusion follows by passing to the infimum over all such coverings of $E$.
\end{proof}

As a preparatory step towards the proof of Theorem~\ref{Thm: Fractional Hardy for HDs} we show a fractional Hardy inequality for test functions with compact support in a domain $\Xi \subseteq \IR^d$ under considerably weaker geometric assumptions than in Proposition~\ref{Prop: Hardy for vanishing boundary trace II}, cf.\ Lemma~\ref{Lem: d-1 set implies boundary density condition}. The price we have to pay is a double integral over $\IR^d$ instead of $\Xi$ on the right-hand side. The proof is by recombining ideas from \cite{Edmunds-Hardy} and~\cite{Hardy-Triebel-LizorkinI}.

\begin{proposition}
\label{Prop: Hardy for testfunctions on Omega-bullet}
Let $0 < 2s < d$ and let $\Xi \subseteq \IR^d$ be a bounded domain with $(s,2)$-uniformly fat complement. Then 
\begin{align*}
 \int_\Xi \frac{\abs{f(x)}^2}{\dist_{\bd \Xi}(x)^{2s}} \; \d x \lesssim \int_{\IR^d} \int_{\IR^d} \frac{\abs{f(x)-f(y)}^2}{\abs{x-y}^{2s+d}} \; \d x \; \d y
\end{align*}
holds for every $f \in \C^\infty(\IR^d)$ with compact support in $\Xi$.
\end{proposition}

\begin{proof}
Let $\mathcal{W}$ be a Whitney decomposition of $\Xi$, i.e.\ $\mathcal{W}$ is a countable family of closed dyadic cubes in $\IR^d$ with pairwise disjoint interiors such that $\Xi = \bigcup_{Q \in \mathcal{W}} Q$ and such that
\begin{align}
\label{Eq: Whitney cubes distance to boundary}
\diam(Q) \leq \mathrm{dist}(Q, \bd \Xi) \leq 4 \diam(Q) \qquad (Q \in \mathcal{W}).
\end{align}
We refer to \cite[Sec.\ VI.1]{Stein} for this classical construction. Denote the center of $Q \in \mathcal{W}$ by $x_Q$ and its side length by $l(Q)$. Let $Q^*:= 40 \sqrt{d} Q$ be the dilated cube having center $x_Q$ and side length $40 \sqrt{d} \cdot l(Q)$, and set $B_{Q^*}:= B(x_Q, c_{d}^{-1} l(Q^*))$ with $c_d > 0$ a constant depending only on $d$; its value to be specified later on.

Now, take $f \in \C^\infty(\IR^d)$ with compact support in $\Xi$. Splitting $\Xi$ into Whitney cubes and employing \eqref{Eq: Whitney cubes distance to boundary} leads to
\begin{align*}
 \int_\Xi \frac{\abs{f(x)}^2}{\dist_{\bd \Xi}(x)^{2s}} \; \d x \leq 2 \sum_{Q \in \mathcal{W}} \diam(Q)^{-2s} \bigg(\abs{Q} \abs{ f_{B_{Q^*}}}^2 + \int_Q \abs{f - f_{B_{Q^*}}}^2 \bigg),
\end{align*}
where $f_{B_{Q^*}}$ denotes the average of $f$ over $B_{Q^*}$. The following average estimates on Whitney cubes of a bounded domain with uniformly fat complement are implicit in the proof of \cite[Thm.~1.3]{Edmunds-Hardy}, see the part following \cite[Eq.\ (4.4)]{Edmunds-Hardy}.
\begin{center}
{\begin{minipage}{0.9 \textwidth}
 \textit{Let $0 < 2s < d$, let $\Xi \subseteq \IR^d$ be a bounded domain with $(s,2)$-uniformly fat complement, and let $\mathcal{W}$ be a Whitney decomposition of \- $\Xi$. There exist constants $c_d > 0$ and $r \in (1,2)$ such that
\begin{align*}
\abs{Q} \abs{f_{B_{Q^*}}}^2 + \int_Q \abs{f - f_{B_{Q^*}}}^2 \lesssim \abs{Q^*}^{2+2s/d - 4/r} \bigg( \int_{Q^*} \int_{Q^*} \frac{\abs{f(x)-f(y)}^r}{\abs{x-y}^{dr/2+rs}} \; \d x \; \d y \bigg)^{2/r}
\end{align*}
 holds for each $f \in \C^\infty(\Xi)$ with compact support in $\Xi$ and every cube $Q \in \mathcal{W}$, where $Q^* = 40 \sqrt{d} Q$ and $B_{Q^*}$ depends on $c_d$  as before.}
\end{minipage}}
\end{center}
Henceforth fix $c_d$ and $r$ suchlike. Next, introduce the auxiliary function $F(x,y):= \frac{\abs{f(x)-f(y)}^r}{\abs{x-y}^{dr/2+rs}}$ and note that $f \in \H^s(\IR^d)$ entails $F \in \L^{2/r}(\IR^d \times \IR^d)$. The combination of the previous two estimates then is
\begin{align*}
      \int_\Xi \frac{\abs{f(x)}^2}{\dist_{\bd \Xi}(x)^{2s}} \; \d x 
      &\lesssim \sum_{Q \in \mathcal{W}} \diam(Q)^{-2s} \abs{Q^*}^{2+2s/d - 4/r} \bigg(\int_{Q^*} \int_{Q^*} F(x,y) \; \d x \; \d y \bigg)^{2/r}
\intertext{and since $Q$ and $Q^*$ are comparable in measure,}
      &\lesssim \sum_{Q \in \mathcal{W}} \abs{Q}^{2} \abs{Q^*}^{- 4/r}\bigg(\int_{Q^*} \int_{Q^*} F(x,y) \; \d x \; \d y \bigg)^{2/r}
      = \sum_{Q \in \mathcal{W}} \abs{Q}^{2} \bigg(\fint_{Q^* \times Q^*} F \bigg)^{2/r}.
\intertext{Now, recall the Hardy-Littlewood Maximal Operator which for $h \in \L_{\loc}^1(\IR^d \times \IR^d)$ is defined by
\begin{align*}
 (\M h)(x,y) := \sup_{Q \in \mathcal{Q}(x,y)} \fint_Q \abs{h} \qquad ((x,y) \in \IR^d \times \IR^d),
\end{align*}
where $\mathcal{Q}(x,y)$ is the collection of closed cubes in $\IR^d \times \IR^d$ that contain a given $(x,y) \in \IR^d \times \IR^d$. By means of $\M$ the ongoing estimate can be continued as follows:}
      \int_\Xi \frac{\abs{f(x)}^2}{\dist_{\bd \Xi}(x)^{2s}} \; \d x &\leq \sum_{Q \in \mathcal{W}} \int_{Q \times Q} \bigg(\fint_{Q^* \times Q^*} F \bigg)^{2/r} \; \d x \; \d y \\
      &\leq \sum_{Q \in \mathcal{W}} \int_{\IR^d \times \IR^d} \ind_{Q \times Q}(x,y) \left(\M F(x,y)\right)^{2/r} \; \d x \; \d y.
\intertext{As the Whitney cubes have pairwise disjoint interiors, $\sum_{Q \in \mathcal{W}} \ind_{Q \times Q} \leq 1$ holds a.e.\ on $\IR^d \times \IR^d$. Monotone convergence and the boundedness of $\M$ on $\L^{2/r}(\IR^d \times \IR^d)$, cf.\ \cite[Thm.\ I.1.1]{Stein}, yield}
      & \leq \int_{\IR^d \times \IR^d} \left(\M F(x,y)\right)^{2/r} \; \d x \; \d y
      \lesssim \int_{\IR^d \times \IR^d} F(x,y)^{2/r} \; \d x \; \d y. 
\end{align*}
By definition of $F$ this completes the proof.
\end{proof}

Surprisingly, Theorem~\ref{Thm: Fractional Hardy for HDs} already follows from Proposition~\ref{Prop: Hardy for testfunctions on Omega-bullet} applied to a very cleverly chosen auxiliary domain $\Omega_\bullet$. This idea is taken from \cite[Sec.\ 6]{ABHR}. 

More precisely, take $\E_\star$ and $\Omega_\star$ as in Theorem~\ref{Thm: Extension operator for HDs}. Recall that $\Omega_\star$ is a bounded domain that contains $\Omega$ and avoids $D$. Let $B \subseteq \IR^d$ be an open ball that contains $\Omega_\star$ and define 
\begin{align*}
\Omega_\bullet := \bigcup \{U: \text{ $U$ is an open and connected subset of $B$ that contains $\Omega$ and avoids $D$} \}.
\end{align*}
Then $\Omega_\bullet$ is a union of domains with a common point and therefore a domain itself. Moreover, $\Omega_\bullet$ is bounded and contains $\Omega_\star$ by construction. Its crucial topological property is the following.

\begin{lemma}[{\cite[Lem.\ 6.4]{ABHR}}]
\label{Lem: Boundary of Omega-bullet}
 It either holds $\bd \Omega_\bullet = D$ or $\bd \Omega_\bullet = \bd B \cup D$.
\end{lemma}

\begin{corollary}
\label{Cor: Omega-bullet has uniformly fat complement}
The complement of $\Omega_\bullet$ is $(s,2)$-uniformly fat for each $1 < 2s < d$.
\end{corollary}

\begin{proof}
By assumption $D$ is a $(d-1)$-set and obviously so is $\bd B$. As a finite union of $(d-1)$-sets $\bd \Omega_\bullet$ is a $(d-1)$-set itself, see Remark~\ref{Rem: Remarks on d-1 set}, and the claim follows from Lemma~\ref{Lem: d-1 set implies boundary density condition}.
\end{proof}

\begin{proof}[Proof of Theorem~\ref{Thm: Fractional Hardy for HDs}]
Let $s \in (\frac{1}{2},1)$ and fix $f \in \C_D^\infty(\Omega)$. Since in any case $D$ is a subset of $\bd \Omega_\bullet$ and as $\E_\star$ is an extension operator,
\begin{align}
\label{Hardy reduction from Omega to Omega-bullet}
 \int_\Omega \frac{\abs{f(x)}^2}{\dist_D(x)^{2s}} \; \d x 
\leq \int_\Omega \frac{\abs{f(x)}^2}{\dist_{\bd \Omega_\bullet}(x)^{2s}} \; \d x 
\leq \int_{\Omega_\bullet} \frac{\abs{\E_\star f(x)}^2}{\dist_{\bd \Omega_\bullet}(x)^{2s}} \; \d x.
\end{align}
Part (4) of Theorem~\ref{Thm: Extension operator for HDs} asserts that the support of the extended function $\E_\star f \in \H_D^s(\IR^d)$ is a subset of $\Omega_\star \subseteq \Omega_\bullet$. Let $\eta$ be a smooth function with support in $\Omega_\bullet$ that is identically one on $\supp(\E_\star f)$. By density choose a sequence $(u_n)_n \subseteq \C_c^\infty(\IR^d)$ that approximates $\E_\star f$ in $\H^s(\IR^d)$.  Lemma~\ref{Lem: Multiplication trace lemma} guarantees that $(\eta u_n)_n$ converges to $\eta \E_\star f = \E_\star f$ in $\H^s(\IR^d)$. After passing to a subsequence we can assume that $(\eta u_n)_n$ converges pointwise a.e.\ on $\IR^d$. Fatou's lemma and Proposition~\ref{Prop: Hardy for testfunctions on Omega-bullet} applied with $\Xi = \Omega_\bullet$ then yield
\begin{align*}
 \int_{\Omega_\bullet} \frac{\abs{\E_\star f(x)}^2}{\dist_{\bd \Omega_\bullet}(x)^{2s}} \; \d x 
&\leq \liminf_{n \to \infty} \int_{\Omega_\bullet} \frac{\abs{\eta(x) u_n(x)}^2}{\dist_{\bd \Omega_\bullet}(x)^{2s}} \; \d x 
\lesssim \liminf_{n \to \infty} \int_{\IR^d} \int_{\IR^d} \! \! \frac{\abs{\eta(x) u_n(x)- \eta(y) u_n(y)}^2}{\abs{x-y}^{2s+d}} \; \d x \; \d y.
\intertext{The rightmost term is bounded by a generic multiple of $\|\eta u_n\|_{\H^s(\IR^d)}^2$. Hence, Theorem~\ref{Thm: Extension operator for HDs} gives}
&\lesssim \liminf_{n \to \infty} \|\eta u_n\|_{\H^s(\IR^d)}
= \|\E_\star f\|_{\H^s(\IR^d)}
\lesssim \|f\|_{\H_D^s(\Omega)}.
\end{align*}
In combination with \eqref{Hardy reduction from Omega to Omega-bullet} this gives the claim of Theorem~\ref{Thm: Fractional Hardy for HDs} for $f \in \C_D^\infty(\Omega)$. 

To establish the claim for general $f \in \H_D^s(\Omega)$, use Proposition~\ref{Prop: Smooth functions dense in HFs} to approximate $f$ in $\H_D^s(\Omega)$ by a sequence $(f_n)_n \subseteq \C_D^\infty(\Omega)$ and conclude by means of Fatou's lemma as before.
\end{proof}

\section{Interpolation Theory}
\label{Sec: Interpolation for HDs}

\noindent This section is devoted to interpolation results related to the spaces $\H_D^s(\Omega)$. There already exists a fully developed interpolation theory for Sobolev spaces that incorporate mixed boundary conditions, cf.\ \cite{Mitrea-PoissonMixed} and \cite{Griepentrog-InterpolationOnGroger}, but -- to our knowledge -- no results obtained so far can cover the very general geometric assumptions on $\Omega$ and $D$ of the present paper.

To begin with, recall the following notions from interpolation theory \cite{Lunardi-Interpolation}, \cite{Triebel}, \cite{Bergh-Loefstroem}. If $X_0$ and $X_1$ are Banach spaces both embedded into the same linear Hausdorff space $\mathcal{X}$ then the spaces $X_0 \cap X_1$ and $X_0 + X_1$ are defined and are complete under the natural norms
\begin{align*}
 \|x\|_{X_0 \cap X_1} &:= \max \big \{\|x\|_{X_0}, \|x\|_{X_1} \big \} \qquad (x \in X_0 \cap X_1), \\
 \|x\|_{X_0 + X_1} &:= \inf \big\{\|x_0\|_{X_0} + \|x_1\|_{X_1}: x_j \in X_j,\ x=x_0 + x_1 \big\} \qquad (x \in X_0 + X_1).
\end{align*}
The pair $(X_0,X_1)$ is called \emph{interpolation couple}. For $\theta \in (0,1)$ the \emph{$\theta$-complex} and the \emph{$(\theta,2)$-real} interpolation space between $X_0$ and $X_1$ are denoted by $[X_0,X_1]_\theta$ and $(X_0, X_1)_{\theta,2}$, respectively. It is convenient to also define these spaces for $\theta \in \{0,1\}$ by setting them equal to $X_\theta$.

The main result we want to show in this section is the following.

\begin{theorem}
\label{Thm: Main interpolation theorem}
Let $\theta \in (0,1)$ and $s_0, s_1 \in (\frac{1}{2},\frac{3}{2})$. In addition, put $s_\theta := (1-\theta)s_0 + \theta s_1$. Then the following hold.
\vspace{5pt}
\begin{enumerate}
\itemsep10pt
 \item $\qquad \big(\H_D^{s_0}(\Omega),\H_D^{s_1}(\Omega)\big)_{\theta,2} = \H_D^{s_\theta}(\Omega) = \big[\H_D^{s_0}(\Omega),\H_D^{s_1}(\Omega)\big]_\theta$.
 \item $\qquad  \big[\L^2(\Omega), \H_D^1(\Omega)\big]_\theta = \big(\L^2(\Omega), \H_D^1(\Omega)\big)_{\theta,2} = \begin{cases}
                                                          \H_D^\theta(\Omega), &\text{if $\theta > \frac{1}{2}$,} \\
							  \H^\theta(\Omega), &\text{if $\theta < \frac{1}{2}$.}
                                                         \end{cases}$
\end{enumerate}
\end{theorem}

\begin{remark}
\label{Rem: Rule of thumb for interpolation}
In combination with reiteration theorems, (2) allows to determine real and complex interpolation spaces between $\H^{s_0}(\Omega)$ and $\H_D^{s_1}(\Omega)$ for $0 \leq s_0 < \frac{1}{2} < s_1 \leq 1$, cf.\ \cite[Sec.\ ~1.10]{Triebel}. Roughly speaking, the trace zero condition on $D$ is maintained under interpolation whenever it is defined, i.e.\ if the resulting Sobolev space has differentiability order larger than $\frac{1}{2}$.
\end{remark}

For the rest of this section the numbers (1) and (2) will refer to the respective items of Theorem~\ref{Thm: Main interpolation theorem}. We can immediately give the purely functorial proof of (1).

\subsection*{Proof of (1)}
If $\frac{1}{2} < s < \frac{3}{2}$ and $D \neq \emptyset$ then $\H_D^s(\IR^d)$ is a complemented subspace of $\H^s(\IR^d)$ in virtue of the projection $\P_D$ introduced in Corollary~\ref{Cor: HFs is complemented}. Thus, by a general result for interpolation of complemented subspaces \cite[Sec.\ 1.17.1]{Triebel}, the set of spaces $\{\H_D^s(\IR^d)\}_{1/2 < s < 3/2}$ interpolates according to the same rules as $\{\H^s(\IR^d)\}_{1/2 < s < 3/2}$. In particular, the well-known interpolation results for Triebel-Lizorkin spaces on $\IR^d$ imply
\begin{align}
\label{Eq: Interpolation for HDs on whole space}
 \big(\H_D^{s_0}(\IR^d),\H_D^{s_1}(\IR^d)\big)_{\theta,2} = \H_D^{s_\theta}(\IR^d) = \big[\H_D^{s_0}(\IR^d),\H_D^{s_1}(\IR^d)\big]_\theta,
\end{align}
see e.g.\ \cite[Sec.\ 2.4.2, Thm.\ 1]{Triebel}. For brevity write $\mathfrak{F}(\H_D^{s_0}(\Omega), \H_D^{s_1}(\Omega))$ for any of the interpolation spaces occurring in (1). With $\E$ the extension operator provided by Theorem~\ref{Thm: Extension operator for HDs}, the retraction-coretraction theorem \cite[Sec.\ 1.2.4]{Triebel} and \eqref{Eq: Interpolation for HDs on whole space} yield that $\E(\H_D^{s_\theta}(\Omega))$ is a closed subspace of $\H_D^{s_\theta}(\IR^d)$ and that
\begin{align*}
 \E: \mathfrak{F}(\H_D^{s_0}(\Omega), \H_D^{s_1}(\Omega)) \to \E(\H_D^{s_\theta}(\Omega))
\end{align*}
is an isomorphism. Thus, $\H_D^{s_\theta}(\Omega)$ and $\mathfrak{F}(\H_D^{s_0}(\Omega), \H_D^{s_1}(\Omega))$ coincide as sets and due to
\begin{align*}
 \|f\|_{\H_D^{s_\theta}(\Omega)} \leq \|\E f\|_{\H_D^{s_\theta}(\IR^d)} \lesssim \|f\|_{\mathfrak{F}(\H_D^{s_0}(\Omega), \H_D^{s_1}(\Omega))} \qquad (f \in \H_D^{s_\theta}(\Omega) )
\end{align*}
and the bounded inverse theorem they also coincide as Banach spaces. This concludes the proof.

\subsection*{Proof of the first equality in (2)}

If $X_0$ and $X_1$ are Hilbert spaces such that $X_0 \subseteq X_1$ with dense and continuous inclusion then $[X_0,X_1]_{\theta} = (X_0,X_1)_{\theta,2}$ holds for each $\theta \in (0,1)$, cf.\ \cite[Cor.\ ~4.37]{Lunardi-Interpolation}. Since in virtue of Corollary~\ref{Cor: Form domain equals HD1} there is an equivalent norm on $\H_D^1(\Omega)$ that is induced by an inner product, the first equality in (2) follows.

\subsection*{Proof of the second equality in (2)}

The second equality in (2) is significantly harder to prove than (1) because the restriction operator $\R_D$, cf.\ Proposition~\ref{Prop: Jonsson-Wallin trace theorem}, is not defined on  $\L^2(\IR^d)$. Our proof relies on a characterization of real interpolation spaces via traces of Banach space valued fractional Sobolev spaces on the real line. Let us recall some notions and properties of these spaces first.

For $X$ a Banach space, $\L^2(\IR;X)$ is the usual Bochner-Lebesgue space of $X$ valued square integrable functions on the real line. For $s > 0$ the respective (fractional) Sobolev spaces $\H^s(\IR;X)$ are defined as in the scalar valued case, cf.\ Section~\ref{Sec: Sobolev mixed boundary}, upon replacing absolute values by norms on $X$. If $s \in \IR_+ \setminus \IN_0$ and $\lfloor s \rfloor$ denotes the integer part of $s$ then
\begin{align}
\label{Vector-valued Sobolev interpolation}
 \big(\H^{\lfloor s \rfloor}(\IR;X), \H^{\lfloor s \rfloor+1}(\IR;X) \big)_{s-\lfloor s \rfloor,2} = \H^s(\IR;X)
\end{align}
by literally the same proof as in \cite[Ex.\ 1.8]{Lunardi-Interpolation}. If $s>\frac{1}{2}$ then each $F \in \H^s(\IR; X)$ has a continuous representative and this gives rise to a continuous inclusion
\begin{align}
\label{BUC embedding}
 \H^s(\IR; X) \subseteq \mathrm{BUC}(\IR;X),
\end{align}
see \cite[Prop.\ 7.4]{Meyries-Veraar-SharpEmbeddings}, or \cite[Thm.\ 5.2]{Grisvard-IntermediateSpaces} for a more direct proof that also applies in the $X$ valued setting. Note that in \cite{Meyries-Veraar-SharpEmbeddings} and \cite{Grisvard-IntermediateSpaces} the spaces $\H^s(\IR; X)$ for non-integer $s$ are defined via \eqref{Vector-valued Sobolev interpolation}. If $s>\frac{1}{2}$ we will, starting from now, identify the elements in $\H^s(\IR; X)$ with their continuous representatives. In virtue of this identification $F \in \H^s(\IR; X)$ can be evaluated at each $t \in \IR$ in a meaningful way. 

The following characterization of real interpolation spaces due to Grisvard \cite[Thm.\ ~5.12]{Grisvard-Interpolation} is of fundamental importance for our further considerations. It gives a description of $(\theta,2)$-real interpolation spaces via traces of $\L^2$ based Sobolev spaces. This will enable us to study these interpolation spaces using the tools from Subsection~\ref{Subsec: Sobolev spaces with partially vanishing traces}.

\begin{theorem}[Grisvard]
\label{Thm: Grisvard trace theorem}
Let the Banach space $X_1$ be densely and continuously included into the Banach space $X_0$ and let $s > \frac{1}{2}$. Then
\begin{align*}
 \big(X_0,X_1\big)_{1-1/(2s), 2} = \big\{ \f_\otimes(0): \f_\otimes \in \L^2(\IR; X_1) \cap \H^s(\IR; X_0) \big\}
\end{align*}
as coinciding sets.
\end{theorem}

The notation used in Theorem~\ref{Thm: Grisvard trace theorem} stems from the fact that in the following $X_0$ and $X_1$ will always be function spaces on $\IR^d$. It is then convenient to identify $\L^2(\IR; X_1) \cap \H^s(\IR; X_0)$ with a function space on $\IR^{d+1}$. More precisely, if for $\f \in \C_c^\infty(\IR^{d+1})$ we put
\begin{align*}
 \f_\otimes: \IR \to \C_c^\infty(\IR^d), \quad t \mapsto \f(t,\argdot),
\end{align*}
where we think of $\IR^{d+1}$ as identified with $\IR \times \IR^d$, then the following holds.

\begin{lemma}
\label{Lem: Identification of functions on tensor product}
If $s \geq 0$ then $\f \mapsto \f_\otimes$ extends by density to a bounded operator from $\H^s(\IR^{d+1})$ into $\L^2(\IR; \H^s(\IR^d)) \cap \H^s(\IR; \L^2(\IR^d))$. This extension is also denoted by $\f \mapsto \f_\otimes$ in the following.
\end{lemma}

\begin{proof}
Recall that $\C_c^\infty(\IR^{d+1})$ is dense in $\H^s(\IR^{d+1})$ for each $s \geq 0$. If $s \in \IN_0$ then Fubini's theorem yields
\begin{align*}
 \|\f_\otimes\|_{\L^2(\IR; \H^s(\IR^d))}^2 +  \|\f_\otimes\|_{\H^s(\IR; \L^2(\IR^d))}^2 \leq \|\f\|_{\H^s(\IR^{d+1})}^2 \qquad (\f \in \C_c^\infty(\IR^{d+1}))
\end{align*}
and the conclusion follows.

Now, assume $s \in \IR_+ \setminus \IN_0$ and put $k:= \lfloor s \rfloor$ and $\theta := s - k$. By the usual interpolation rules for Triebel-Lizorkin spaces, see e.g.\ \cite[Sec.\ 2.4.2, Thm.\ 1]{Triebel},
\begin{align}
\label{Identification of functions on tensorproduct: Eq1}
 \big(\H^k(\IR^{d+1}), \H^{k +1}(\IR^{d+1})\big)_{\theta ,2} = \H^s(\IR^{d+1}) =  \big[\H^k(\IR^{d+1}), \H^{k +1}(\IR^{d+1})\big]_\theta .
\end{align}
Hence, $(\theta,2)$-real and $\theta$-complex interpolation of the claims for $k$ and $k+1$ show that $\f \mapsto \f_\otimes$ acts as a bounded operator from $\H^s(\IR^{d+1})$ into both
\begin{align*}
 \big(\H^k(\IR;\L^2(\IR^d)), \H^{k+1}(\IR;\L^2(\IR^d)) \big)_{\theta,2} \quad \text{and} \quad \big[\L^2(\IR; \H^k(\IR^d)), \L^2(\IR; \H^{k + 1}(\IR^d)) \big]_\theta.
\end{align*}
To conclude, note that by \eqref{Vector-valued Sobolev interpolation} the left-hand space equals $\H^s(\IR; \L^2(\IR^d))$, whereas the right-hand space can be revealed as $\L^2(\IR; \H^s(\IR^d))$ using the interpolation rule
\begin{align*}
 \big[\L^2(\IR; X_0), \L^2(\IR; X_1)\big]_\theta = \L^2\big(\IR; \big[X_0,X_1\big]_\theta\big),
\end{align*}
see \cite[Thm.\ 5.1.2]{Bergh-Loefstroem} for details, and applying \eqref{Identification of functions on tensorproduct: Eq1} for function spaces on $\IR^d$. 
\end{proof}

As a technical tool we need the following property of $l$-sets. To distinguish objects in $\IR^{d+1}$ from their counterparts in $\IR^d$ we shall keep on using bold letters for the former.

\begin{lemma}
\label{Lem: Tensoring d-sets}
Let $0 < l \leq d$. If $E \subseteq \IR^d$ is an $l$-set and $I \subseteq \IR$ is an interval that is not reduced to a single point, then $I \times E$ is an $(l+1)$-set in $\IR^{d+1}$.
\end{lemma}

\begin{proof}
First note that for $(t,x) \in I \times E$ and $r >0$ it holds
\begin{align}
\label{Eq: Tensoring d-sets}
 (t-r,t+r) \times B(x,r) \subseteq \B((t,x),2r) \subseteq (t-2r,t+2r) \times B(x,2r).
\end{align}
It is a classical result that $\bm_{l+1}(U \times V) \simeq \abs{U} \cdot  \m_l(V)$ holds with implicit constants depending only on $d$ provided that $U \subseteq \IR$ is Lebesgue measurable and $V \subseteq \IR^d$ has finite $\m_l$-measure, see e.g.\ \cite[Thm.\ 2.10.45]{Federer}. Thus, intersecting the inclusions in \eqref{Eq: Tensoring d-sets} with $I \times E$ leads to
\begin{align*}
 \bm_{l+1}\big((I \times E) \cap \B((t,x),2r)\big) \simeq r^{l+1} \qquad ((t,x) \in I \times E,\, 2r < 1).
\end{align*}
By Remark~\ref{Rem: Remarks on d-1 set} this concludes the proof. 
\end{proof}

\begin{corollary}
\label{Cor: Omega-Odot-D is a d-set}
The \emph{infinite $D$ cylinder} $\Dcyl:= (\{0\} \times \Omega) \cup (\IR \times D)$ is a $d$-set in $\IR^{d+1}$.
\end{corollary}

\begin{proof}
If $D \neq \emptyset$ then Lemma~\ref{Lem: Tensoring d-sets} asserts that $\IR \times D$ is a $d$-set in $\IR^{d+1}$. Hence, the conclusion follows by Remarks~\ref{Rem: Remarks on d-1 set} and \ref{Rem: Remark on geometric assumptions}.
\end{proof}

Our next result shows that functions on $\Omega$ can be trivially extended to $\Dcyl$ without losing Sobolev regularity. Here, the fractional Hardy type inequality from Section~\ref{Sec: Fractional Hardy for HDs} comes into play.

\begin{proposition}
\label{Prop: Zero extension to Omega-Odot-D}
Let $s \in (\frac{1}{2},1)$ and $f \in \H_D^s(\Omega)$. Then the function
\begin{align*}
 f_\uparrow:  \Dcyl \to \IC, \qquad f_\uparrow(t,x) = \begin{cases}
				    f(x), & \text{if $t=0$, $x\in \Omega$,} \\
				    0, & \text{if $x \in D$,}
                                 \end{cases}
\end{align*}
is in $\H^s(\Dcyl,\, \bm_d)$, where $\bm_d$ is the $d$-dimensional Hausdorff measure in $\IR^{d+1}$, and satisfies the estimate $\|f_\uparrow\|_{\H^s(\Omega \uparrow D,\, \bm_d)} \lesssim \|f\|_{\H^s(\Omega)}$. A similar statement holds if $s \in (0,\frac{1}{2})$ and $f \in \H^s(\Omega)$.
\end{proposition}

\begin{proof}
Let $s \in (\frac{1}{2},1)$. Since the outer measure $E \mapsto \bm_d(\{0\} \times E)$ on $\IR^d$ is a translation invariant Borel measure that assigns finite measure to the unit cube, the induced measure coincides up to a norming constant $c_d > 0$ with the $d$-dimensional Lebesgue measure, see e.g.\ \cite[Thm.\ 8.1]{Bauer-Masstheorie}. Thus, $f_\uparrow \in \L^2(\Dcyl,\, \bm_d)$ is a consequence of $f \in \L^2(\Omega)$. 

To compute the complete $\H^s(\Dcyl, \bm_d)$ norm of $f_\uparrow$, split integration over $(\Dcyl) \times (\Dcyl)$ according to the definition of $f_\uparrow$ and use Tonelli's theorem to find
\begin{align}
\label{Zero extension to Omega-Odot-D: Eq1}
\begin{split}
&\iint_{\substack{\x, \y \in \Omega \uparrow D \\ \abs{\x-\y}<1}} \frac{\abs{f_\uparrow(\x) - f_\uparrow(\y)}^2}{\abs{\x-\y}^{d+2s}} \; \d \bm_d(\x) \; \d \bm_d(\y) \\
&\leq c_d \iint_{\substack{x,y \in \Omega \\ \abs{x-y}<1}} \frac{\abs{f(x) - f(y)}^2}{\abs{x-y}^{d+2s}} \; \d x \; \d y
+2 \int_{\{0\} \times \Omega} \int_{\substack{\x \in \IR \times D \\ \abs{\x - \y}<1}} \frac{\abs{f_\uparrow(\y)}^2}{\abs{\x - \y}^{d+2s}} \; \d \bm_d(\x) \; \d \bm_d(\y).
\end{split}
\end{align}
The first integral on the right-hand side is bounded by $\|f\|_{\H_D^s(\Omega)}^2$. To handle the second one fix $\y = (0,y) \in \{0\} \times \Omega$. If the inner domain of integration is non-empty then there exists an $n_0 \in \IN_0$ such that $2^{-(n_0+1)} < \bdist(\y, \IR \times D) < ~2^{-n_0}$. Splitting the integral into frame-like pieces 
\begin{align*}
 \mathbf{C}_n:= \big(\IR \times D\big) \cap \big((\B(\y, 2^{-n}) \setminus \B(\y, 2^{-(n+1)})\big) \qquad (0 \leq n \leq n_0)
\end{align*}
leads to
\begin{align*}
 \int_{\substack{\x \in \IR \times D \\ \abs{\x - \y}<1}} \frac{1}{\abs{\x - \y}^{d+2s}} \; \d \bm_d(\x) 
&\leq \sum_{n=0}^{n_0} 2^{(n+1)(d+2s)} \bm_d(\mathbf{C}_n) \lesssim \sum_{n=0}^{n_0} 2^{(n+1)(d+2s)} 2^{-dn},
\intertext{where the second step follows since $\Dcyl$ is a $d$-set in $\IR^{d+1}$. An explicit computation gives}
&=\frac{2^{d+2s}}{2^{2s}-1} (2^{2s(n_0+1)} -1) \lesssim \bdist(\y, \IR \times D)^{-2s} = \d(y,D)^{-2s}
\end{align*}
with implicit constants depending solely on $d$ and $s$. Now, Theorem~\ref{Thm: Fractional Hardy for HDs} allows to estimate
\begin{align*}
 \int_{\substack{\x \in \IR \times D \\ \abs{\x - \y}<1}} \frac{\abs{f_\uparrow(\y)}^2}{\abs{\x - \y}^{d+2s}} \; \d \bm_d(\x) \; \d \bm_d(\y)
\lesssim \int_\Omega \frac{\abs{f(y)}^2}{\d_D(y)^{2s}} \; \d y \lesssim \|f\|_{\H_D^s(\Omega)}^2.
\end{align*}
With a view on \eqref{Zero extension to Omega-Odot-D: Eq1} this completes the proof in the case $s>\frac{1}{2}$. 

If $s < \frac{1}{2}$ the argument is literally the same except that we can simply rest on Proposition~\ref{Prop: Hardy for vanishing boundary trace I} instead of Theorem~\ref{Thm: Fractional Hardy for HDs}, noting that $\d_D(y) \geq \d_{\bd \Omega}(y)$ for each $y \in \Omega$.
\end{proof}

We have now collected all necessary tools to establish the second equality in (2). The challenge is, as it turns out, to determine \emph{any} interpolation space between $\L^2(\Omega)$ and a Sobolev space incorporating mixed boundary conditions in the first place. This is done in the subsequent proposition. The actual proof can then be completed using reiteration techniques.

\begin{proposition}
\label{Prop: Interpolation with non-matching parameters}
If $s \in (0,1)$ and $\vartheta = \frac{2}{2s+1}$ then
\begin{align*}
 \big(\L^2(\Omega), \H_D^{s+1/2}(\Omega)\big)_{\vartheta s,2} = \begin{cases}
                                                          \H_D^s(\Omega), &\text{if $s > \frac{1}{2}$,} \\
							  \H^s(\Omega), &\text{if $s < \frac{1}{2}$.}
                                                         \end{cases}
\end{align*}
\end{proposition}

\begin{proof}
We prove both continuous inclusions separately.

$\subseteq$ \,:
For brevity put $X:= (\L^2(\Omega), \H_D^{s+1/2}(\Omega))_{\vartheta s,2}$. Let $\E$ be the extension operator provided by Theorem~\ref{Thm: Extension operator for HDs}. By $(\vartheta s,2)$-real interpolation and the interpolation rules for Triebel-Lizorkin spaces \cite[Sec.\ 2.4.2, Thm.\ 1]{Triebel}, $\E$ is bounded from $X$ into
\begin{align}
\label{Interpolation with non-matching parameters: Eq1}
 \big(\L^2(\IR^d), \H_D^{s+1/2}(\IR^d)\big)_{\vartheta s,2} \subseteq \big(\L^2(\IR^d), \H^{s+1/2}(\IR^d)\big)_{\vartheta s,2} = \H^s(\IR^d).
\end{align}
To see that $\E$ in fact maps into $\H_D^s(\IR^d)$ if $D \neq \emptyset$ and $s > \frac{1}{2}$, first note that in this case $\vartheta s \in (\frac{1}{2}, 1)$. Hence, it is possible to find $\lambda \in (\frac{1}{2}, \vartheta s)$ and $\gamma \in (0,1)$ such that $\vartheta s = (1-\gamma) \lambda + \gamma$. The reiteration theorem for real interpolation \cite[Sec.\ 1.10.2]{Triebel} yields 
\begin{align*}
 \E(X)\subseteq \big(\L^2(\IR^d), \H_D^{s+1/2}(\IR^d) \big)_{\vartheta s,2} = \big(\big(\L^2(\IR^d), \H_D^{s+1/2}(\IR^d)\big)_{\lambda,2}, \H_D^{s+1/2}(\IR^d)\big)_{\gamma,2} =: \big(Y_0,Y_1\big)_{\gamma,2}.
\end{align*}
As in \eqref{Interpolation with non-matching parameters: Eq1} it follows that $Y_0$ is continuously included in $\H^{\lambda (s+1/2)}(\IR^d)$. Due to $\lambda (s+\frac{1}{2}) > \frac{1}{2}$ the restriction operator $\R_D$ from Proposition~\ref{Prop: Jonsson-Wallin trace theorem} is defined on both $Y_0$ and $Y_1$, mapping them into the respective Sobolev spaces on $D$. But, by definition, $Y_1$ is contained in the null space of $\R_D$. Since $(\gamma,2)$-real interpolation is exact of type $\gamma$, see \cite[Sec.\ 1.3.3]{Triebel} for details, $(Y_0,Y_1)_{\gamma,2}$ and hence $\E(X)$ is contained in the null space of $\R_D$ as well. Due to \eqref{Interpolation with non-matching parameters: Eq1} this implies $\E(X) \subseteq \H_D^s(\IR^d)$.

From the considerations above we conclude that if $s> \frac{1}{2}$ then each $f \in X$ belongs to $\H_D^s(\Omega)$ as the restriction of $\E f \in \H_D^s(\IR^d)$ and that, since $\E: X \to \H_D^s(\IR^d)$ is bounded, this inclusion is continuous. Likewise, if $s< \frac{1}{2}$ then $X \subseteq \H^s(\Omega)$ with continuous inclusion.

\medskip

$\supseteq $ \,: We concentrate on the case $s > \frac{1}{2}$. Upon replacing $\H_D^s(\Omega)$ by $\H^s(\Omega)$ the proof in the case $s < \frac{1}{2}$ is literally the same. The roadmap for the somewhat involved argument reads as follows:
\begin{align*}
\\[-5pt]
\begin{CD}
\H_{\IR \times D}^{s+1/2}(\IR^{d+1}) @>{\text{Lem. \ref{Lem: Identification of functions on tensor product}}}>> \L^2\big(\IR; \H_D^{s+1/2}(\IR^d)\big) \cap \H^{s+1/2}\big(\IR; \L^2(\IR^d)\big)\\
@A{\E_{\Omega \uparrow D}}AA @VV{\R_\Omega}V \\
\H^s(\Dcyl) @. \L^2\big(\IR; \H_D^{s+1/2}(\Omega)\big) \cap \H^{s+1/2}\big(\IR; \L^2(\Omega)\big) \\
@A{\text{Prop. \ref{Prop: Zero extension to Omega-Odot-D}}}AA @VV{\text{Thm.\ \ref{Thm: Grisvard trace theorem}}}V \\
\H_D^s(\Omega) @. \big(\L^2(\Omega), \H_D^{s+1/2}(\Omega)\big)_{\vartheta s,2}.
\end{CD}
\\[-5pt]
\end{align*}

To make this precise, first note that in view of Theorem~\ref{Thm: Grisvard trace theorem} and the bounded inverse theorem it suffices to construct for general $f \in \H_D^s(\Omega)$ a function $\f_\otimes$ such that
\begin{align}
\label{Interpolation with non-matching parameters: desired properties of F}
 \f_\otimes \in \L^2\big(\IR; \H_D^{s+1/2}(\Omega)\big) \cap \H^{s+1/2}\big(\IR; \L^2(\Omega)\big), \quad \f_\otimes(0) = f.
\end{align}
For the construction let $f_\uparrow \in \H^s(\Dcyl,\, \bm_d)$ be given by Proposition~\ref{Prop: Zero extension to Omega-Odot-D}. Apply Proposition~\ref{Prop: Jonsson-Wallin trace theorem} to the $d$-set $\Dcyl \subseteq \IR^{d+1}$ to obtain an extension $\g \in \H^{s+1/2}(\IR^{d+1})$ of $f_\uparrow$. In virtue of Lemma~\ref{Lem: Identification of functions on tensor product} this extension is related to the the Banach space valued function
\begin{align*}
 \g_\otimes \in \L^2\big(\IR; \H^{s+1/2}(\IR^d)\big) \cap \H^{s+1/2}\big(\IR; \L^2(\IR^d)\big).
\end{align*}
A closer inspection of $\g_\otimes$ making use of the exact definition of $f_\uparrow$ reveals the following.
\begin{enumerate}[(i)]
\itemsep3pt
  \item By definition of $f_\uparrow$ it holds $\g \in \H_{\IR \times D}^{s+1/2}(\IR^{d+1}) \subseteq \H_{\IR \times D}^1(\IR^{d+1})$. Note that this notation is meaningful for $\IR \times D$ is either empty or a $d$-set in $\IR^{d+1}$ thanks to Lemma~\ref{Lem: Tensoring d-sets}. Proposition~\ref{Prop: Smooth functions dense in HFs} provides a sequence $(\g_n)_n$ of smooth, compactly supported functions whose support avoids $\IR \times D$ and that approximates $\g$ in $\H^1(\IR^{d+1})$. Owing to Lemma~\ref{Lem: Identification of functions on tensor product} we can, after passing to a suitable subsequence, assume for almost all $t \in \IR$ that
  \begin{align*}
    \qquad \lim_{n\to \infty} \g_n (t, \argdot) = \lim_{n \to \infty} (\g_n)_\otimes(t) = \g_\otimes(t) \qquad (\text{in $\H^1(\IR^d)$}).
  \end{align*}
  Since $\g_n(t,\argdot) \in \C_D^\infty(\IR^d)$ holds for all $t \in \IR$ by construction, this entails that for a.e.\ $t \in \IR$ the function $\g_\otimes(t) \in \H^{s+1/2}(\IR^d)$ satisfies $\R_D (\g_\otimes(t)) = 0$, i.e.\ belongs to the closed subspace $\H_D^{s+1/2}(\IR^d)$. Here, $\R_D$ is the restriction operator to the $(d-1)$-set $D$, cf.\ Proposition~\ref{Prop: Jonsson-Wallin trace theorem}, and we have used its boundedness from $\H^1(\IR^d)$ onto $\L^2(D, \m_{d-1})$. Summing up, it follows $\g_\otimes \in \L^2\big(\IR; \H_D^{s+1/2}(\IR^d)\big)$.
 
  \item Lemma~\ref{Lem: Identification of functions on tensor product} in combination with the embedding \eqref{BUC embedding} reveals $\g_\otimes(0)$ as the $\L^2(\IR^d)$-limit of $(\g_n(0,\argdot))_n$. But as $\{0\} \times \Omega$ is a $d$-set in $\IR^{d+1}$, cf.\ Remark~\ref{Rem: Remark on geometric assumptions}, Proposition~\ref{Prop: Jonsson-Wallin trace theorem} provides a bounded restriction operator $\R_{\{0\} \times \Omega}: \H^1(\IR^{d+1}) \to \L^2(\{0\} \times \Omega,\, \bm_d)$ and it also follows
  \begin{align*}
    \qquad \lim_{n \to \infty} \g_n|_{\{0\} \times \Omega} = \lim_{n \to \infty} \R_{\{0\} \times \Omega} (\g_n) = \R_{\{0\} \times \Omega}(\g) = f_\uparrow|_{\{0\} \times \Omega}   \qquad (\text{in $\L^2(\{0\} \times \Omega,\, \bm_d)$}).
  \end{align*}
  Identifying the measure spaces $(\Omega,\, \abs{\argdot})$ and $(\{0\} \times \Omega,\, \bm_d)$ as in the proof of Proposition~\ref{Prop: Zero extension to Omega-Odot-D} we conclude from the previous observations that $\g_\otimes(0) = f$ holds a.e.\ on $\Omega$. 
\end{enumerate}
Altogether,
\begin{align*}
 \g_\otimes \in \L^2\big(\IR; \H_D^{s+1/2}(\IR^d)\big) \cap \H^{s+1/2}\big(\IR; \L^2(\IR^d)\big), \quad \g_\otimes(0)|_{\Omega} = f,
\end{align*}
so that \eqref{Interpolation with non-matching parameters: desired properties of F} holds for the choice $\f_\otimes(t) := \g_\otimes(t)|_{\Omega}$, $t \in \IR$.
\end{proof}

Now, the proof of the second equality in (2) can easily be completed. In the following all function spaces will be on $\Omega$, so for brevity we shall write $\L^2$ instead of $\L^2(\Omega)$ and so on. We have to show
\begin{align*}
 \big(\L^2, \H_D^1 \big)_{s,2} = \H_D^s \quad \text{and} \quad \big(\L^2, \H_D^1 \big)_{t,2} = \H^t \qquad(0 < t < \tfrac{1}{2} < s <1).
\end{align*}
Given $s \in (\frac{1}{2},1)$ set $\vartheta: = \frac{2}{2s+1}$. Observe that $\vartheta s < \vartheta < 1$ so that there exists a $\lambda  \in (0,1)$ such that $\vartheta = (1-\lambda) \vartheta s + \lambda$. Using in sequence the reiteration theorem for real interpolation, cf.\ \cite[Sec.\ ~1.10.2]{Triebel}, Proposition~\ref{Prop: Interpolation with non-matching parameters}, and (1) in Theorem~\ref{Thm: Main interpolation theorem} leads to
\begin{align*}
 \big(\L^2, \H_D^{s+1/2}\big)_{\vartheta,2} 
= \big( \big(\L^2, \H_D^{s+1/2}\big)_{\vartheta s,2},  \H_D^{s+1/2} \big)_{\lambda,2}
= \big(\H_D^s, \H_D^{s+1/2}\big)_{\lambda,2}
= \H_D^1.
\end{align*}
Reapplication of the reiteration theorem and Proposition~\ref{Prop: Interpolation with non-matching parameters} yield the desired equality
\begin{align}
\label{L2-HD1 interpolation large s}
\big(\L^2, \H_D^1 \big)_{s,2} = \big(\L^2, \big(\L^2, \H_D^{s + 1/2} \big)_{\vartheta,2} \big)_{s,2} =  \big(\L^2, \H_D^{s+1/2}\big)_{\vartheta s,2} = \H_D^s.
\end{align}
Likewise for $t \in (0,\frac{1}{2})$ set $\vartheta := \frac{2}{2t+1}$ and employ in sequence the reiteration theorem, \eqref{L2-HD1 interpolation large s} for the choice $s=t+\frac{1}{2}$, and Proposition~\ref{Prop: Interpolation with non-matching parameters} to find
\begin{align*}
 \big(\L^2, \H_D^1 \big)_{t,2} = \big(\L^2, \big(\L^2, \H_D^1 \big)_{t+1/2,2}\big)_{\vartheta t,2} = \big(\L^2, \H_D^{t+1/2}\big)_{\vartheta t,2} = \H^t
\end{align*}
and the proof is complete. \hfill $\square$

\subsection*{A remark on the critical case $\theta = \frac{1}{2}$} As the trace operator $\R_D$ from Proposition~\ref{Prop: Jonsson-Wallin trace theorem} is only defined on $\H^\theta(\IR^d)$ if $\theta> 1/2$, there is no analogously defined space $\H_D^{1/2}(\Omega)$. Still, of course, there are $(\frac{1}{2},2)$-real and $\frac{1}{2}$-complex interpolation spaces between $\L^2(\Omega)$ and $\H_D^1(\Omega)$ and the question arises if these spaces know about the trace zero condition on $D$ in any reasonable sense. The following proposition shows that the respective interpolation spaces can be characterized by a fractional Hardy type inequality.

\begin{proposition}
\label{Prop: Interpolation in the critical case}
The interpolation spaces $(\L^2(\Omega), \H_D^1(\Omega))_{1/2, 2}$ and  $[\L^2(\Omega), \H_D^1(\Omega)]_{1/2}$ both coincide with $\H^{1/2}(\Omega) \cap \L^2(\Omega, \d x / \dist_D(x))$, i.e.\ the space of all $f \in \H^{1/2}(\Omega)$ such that
\begin{align*}
 \int_\Omega \frac{\abs{f(x)}^2}{\dist_D(x)} \; \d x  < \infty,
\end{align*}
equipped with its natural norm.
\end{proposition}

\begin{proof}
For brevity put $X:= (\L^2(\Omega), \H_D^1(\Omega))_{1/2,2}$ and  $Y:= \H^{1/2}(\Omega) \cap \L^2(\Omega, \d x / \dist_D(x))$. First, recall from the proof of part (1) of Theorem~\ref{Thm: Main interpolation theorem} that $X = [\L^2(\Omega), \H_D^1(\Omega)]_{1/2}$.

In order to prove $X = Y$, first let $f \in Y$. Then $f_\uparrow$ defined in Proposition~\ref{Prop: Zero extension to Omega-Odot-D} belongs to $\H^{1/2}(\Dcyl,\, \bm_d)$. Indeed, in the proof of Proposition~\ref{Prop: Zero extension to Omega-Odot-D} the restriction $s>\frac{1}{2}$ has only been used in the very last estimate in order to guarantee that $\int_\Omega \abs{f}^2 \dist_D^{-2s}$ is finite. For $f \in Y$ and $s = \frac{1}{2}$ this, however, follows by definition of $Y$. Therefore $f \in X$ follows literally as in part `$\supseteq$' of the proof of Proposition~\ref{Prop: Interpolation with non-matching parameters}.

The next step is to prove $X \subseteq \H^{1/2}(\Omega)$ with continuous inclusion. To this end, let $\E$ be the extension operator provided by Theorem~\ref{Thm: Extension operator for HDs}. Using the classical interpolation result $(\L^2(\IR^d), \H^1(\IR^d))_{1/2,2} = \H^{1/2}(\IR^d)$, see e.g.\ \cite[Sec.~2.4.2, Thm.~1]{Triebel}, it follows that $\E$ maps $X$ boundedly into $\H^{1/2}(\IR^d)$. Since the restriction from $\H^{1/2}(\IR^d)$ onto $\H^{1/2}(\Omega)$ is bounded, this yields the claim. 

Finally, it remains to prove $X \subseteq \L^2(\Omega, \d x / \dist_D(x))$ with continuous inclusion. Here, note that in virtue of the Hardy type inequality
\begin{align*}
 \int_\Omega \frac{\abs{f(x)}^2}{\dist_D(x)^2} \; \d x \lesssim \|f\|_{\H_D^1(\Omega)} \qquad (f \in \H_D^1(\Omega))
\end{align*}
from \cite[Thm.~6.1]{ABHR} there is a continuous inclusion $\H_D^1(\Omega) \subseteq \L^2(\Omega, \d x / \dist_D(x)^2)$. Hence, the claim follows by $(\frac{1}{2},2)$-real interpolation of $\L^2$ spaces with a change of measure \cite[Thm.~5.4.1]{Bergh-Loefstroem}.
\end{proof}

\begin{remark}
\label{Rem: No knowledge on critical case}
Unlike in the case $s \in (0,\frac{1}{2})$, the fractional Hardy inequality occurring above encapsulates some boundary behavior on $D$ and thus is not satisfied by every $f \in \H^{1/2}(\Omega)$, cf.\ Proposition~\ref{Prop: Hardy for vanishing boundary trace I}. For example, let $\Omega := B(0,1)$, $D := \partial B(0,1)$, and $f :\equiv 1$. Then of course $f \in \H^{1/2}(\Omega)$ but $\int_\Omega \abs{f(x)}^2\dist_D(x)^{-1} \; \d x \simeq \int_0^1 r^{d-1} (1-r)^{-1} \; \d r = \infty$. This also shows that the upper bound for the range of exponents in Proposition~\ref{Prop: Hardy for vanishing boundary trace I} is sharp.
\end{remark}

\section{Proof of the Main Result}
\label{Sec: Proof of the main result}

\noindent We now turn to the proof of our main result, Theorem~\ref{Thm: Main Result}. Again $\Delta_\V$ denotes the weak Laplacian with form domain $\V$, cf.\ Subsection~\ref{Subsec: The Elliptic Operator}. Then $1-\Delta_\V$ is an invertible, maximal accretive self-adjoint operator on $\L^2(\Omega)$ with associated sesquilinear form
\begin{align*}
\j: \V \times \V \to \IC,  \quad \j(u,v) = \int_\Omega u \cdot \cl{v}  + \int_\Omega \nabla u \cdot \nabla \cl{v}.
\end{align*}
Recall by Corollary~\ref{Cor: Form domain equals HD1} and the square root property for self-adjoint operators \cite[Thm.\ ~VI.2.23]{Kato} that
\begin{align*}
 \H_D^1(\Omega) = \V = \dom((1-\Delta_\V)^{1/2})
\end{align*}
holds up to equivalent norms. Starting from this we obtain
\begin{align}
\label{Eq: Fractional power domains under square root}
 \dom((-\Delta_\V)^\alpha) = \dom((1-\Delta_\V)^\alpha) = \big[\L^2(\Omega), \H_D^1(\Omega)\big]_{2 \alpha} 
= \begin{cases}
   \H_D^{2 \alpha}(\Omega), & \text{if $\alpha \in (\tfrac{1}{4}, \tfrac{1}{2}]$,} \\
   \H^{2 \alpha}(\Omega), & \text{if $\alpha \in [0,\tfrac{1}{4})$}
  \end{cases}
\end{align}
thanks to Theorem~\ref{Thm: Main interpolation theorem} and the following classical result for maximal accretive operators.

\begin{proposition}[{\cite[Cor.\ 4.30]{Lunardi-Interpolation}}]
\label{Prop: BIP Lemma}
If $B$ is an invertible, maximal accretive operator on a Hilbert space, then for all $\alpha, \beta \geq 0$ and for all $\theta \in [0,1]$ it holds
\begin{align*}
 \big[\dom(B^\alpha), \dom(B^\beta)\big]_\theta = \dom(B^{(1-\theta)\alpha + \theta \beta}).
\end{align*}
\end{proposition}

In view of \eqref{Eq: Fractional power domains under square root} it remains to show that there exists an $\eps \in (0, \frac{1}{4})$ such that
\begin{align}
\label{Eq: Last goal}
 \dom((1-\Delta_\V)^\alpha) = \H_D^{2 \alpha}(\Omega) \qquad (\alpha \in (\tfrac{1}{2}, \tfrac{1}{2} + \eps)).
\end{align}
Here we used again that the domains of the respective fractional powers of $-\Delta_\V$ and $1-\Delta_\V$ coincide. 

We will establish \eqref{Eq: Last goal} by means of an interpolation argument going back to Pryde \cite{Pryde-MixedBoundary}, see also \cite{AKM}. Throughout, $X^*$ denotes the \emph{anti dual space} of a Banach space $X$, i.e.\ the space of all bounded conjugate linear functionals on $X$. Occasionally, we apply results on dual spaces also in the anti dual setting. These arguments can all be justified by the simple observation that $x^*$ is an element of $X^*$, if and only if its conjugate $\cl{x^*}$ is in the dual of $X$. 

All function spaces occurring in the following will be on $\Omega$, so for brevity we shall again write $\L^2$ instead of $\L^2(\Omega)$ and so on. We begin with the following interpolation estimates for $\j$.

\begin{lemma}
\label{Lem: First Interpolation estimates for j}
If $\alpha \in [\frac{1}{2},\frac{3}{4})$ then
\begin{align*}
 \abs{\j(u,v)} \lesssim \|u\|_{\dom((1-\Delta)^\alpha)} \|v\|_{\H_D^{2-2\alpha}} \qquad (u \in \dom((1-\Delta_\V)^\alpha),\, v \in \V).
\end{align*}
\end{lemma}

\begin{proof}
Since $\dom(1-\Delta_\V)$ is a core for $\dom((1-\Delta_\V)^\alpha)$ and since the latter is continuously included into $\dom((1-\Delta_\V)^{1/2}) = \V$ it suffices, by approximation, to consider the special case $u \in \dom(1-\Delta_\V)$. As with $1-\Delta_\V$ also its fractional powers are self-adjoint, cf.\ \cite[Prop.\ 2.6.3]{Haase}, it follows
\begin{align*}
 \big|\j(u,v) \big| 
&= \big|\big \langle (1-\Delta_\V)u, v \big \rangle_{\L^2}\big| 
= \big|\big \langle (1-\Delta_\V)^\alpha u, (1-\Delta_\V)^{1-\alpha}v \big \rangle_{\L^2}\big| \\
&\leq \|u \|_{\dom((1-\Delta_\V)^\alpha)} \|v \|_{\dom((1-\Delta_\V)^{1-\alpha})}
\end{align*}
for all $v \in \V$. This already yields the claim since $\dom((1-\Delta_\V)^{1-\alpha}) = \H_D^{2-2\alpha}$ holds up to equivalent norms thanks to \eqref{Eq: Fractional power domains under square root}.
\end{proof}

\begin{lemma}
\label{Lem: Second Interpolation estimates for j}
If $\alpha \in (\frac{1}{4},\frac{1}{2}]$ then
\begin{align*}
 \abs{\j(u,v)} \lesssim \|u\|_{\H_D^{2\alpha}} \|v\|_{\H_D^{2-2 \alpha}} \qquad (u \in \V,\, v \in \H_D^{2-2 \alpha}).
\end{align*}
\end{lemma}

\begin{proof}
Recall from Remark~\ref{Rem: Remark on geometric assumptions} that $\bd \Omega$ is a $(d-1)$-set. Hence, if the pair $(\Omega, D)$ satisfies Assumption~\ref{Ass: General geometric assumption on Omega} then so does $(\Omega, \bd \Omega)$. Therefore, Theorem~\ref{Thm: Main interpolation theorem} combined with a duality principle for complex interpolation \cite[Cor.\ 4.5.2]{Bergh-Loefstroem} yields the interpolation identities
\begin{align}
\label{Eq: Second Interpolation estimates for j}
 \big[\L^2, \H_D^1\big]_{2\alpha} = \H_D^{2\alpha} \quad \text{and} \quad \big[(\L^2)^*,(\H_{\bd \Omega}^1)^*\big]_{1-2\alpha} = \big[\L^2,\H_{\bd \Omega}^1\big]_{1-2\alpha}^* = (\H^{1-2 \alpha})^*.
\end{align}

Let $1 \leq j \leq d$. By Proposition~\ref{Prop: Smooth functions dense in HFs} the test function space $\C_c^\infty(\Omega)$ is dense in $\H_{\bd \Omega}^1$. Given $f \in \L^2$, the distributional derivative $\partial_j f$ can therefore be canonically regarded as an element of $(\H_{\bd \Omega}^1)^*$. In virtue of this identification 
\begin{align*}
 \partial_j: \big[\L^2, \H_D^1\big]_{2\alpha} 
\to  \big[(\H_{\bd \Omega}^1)^*, (\L^2)^*\big]_{2\alpha} = \big[(\L^2)^*,(\H_{\bd \Omega}^1)^*\big]_{1-2\alpha}
\end{align*}
is bounded. Taking \eqref{Eq: Second Interpolation estimates for j} into account we conclude that $\partial_j$ maps $\H_D^{2\alpha}$ boundedly into $(\H^{1-2 \alpha })^*$.

To establish the actual claim, simply note that $\partial_j$ also maps $\H_D^{2-2 \alpha}$ boundedly into $\H^{1-2\alpha}$, where this time distributional derivatives are identified with $\L^2$ functions rather than with functionals, and conclude for $u \in \V$ and $v \in \H_D^{2-2 \alpha}$ the desired estimate
\begin{align*}
 \abs{\j(u,v)} \leq \|u\|_{\L^2} \|v\|_{\L^2} + \sum_{j = 1}^d \|\partial_j u\|_{(\H^{1-2 \alpha})^*} \|\partial_j v\|_{\H^{1-2\alpha}} \lesssim \|u\|_{\H_D^{ 2\alpha}} \|v\|_{\H_D^{2-2 \alpha}}. & \qedhere
\end{align*}
\end{proof}

Our main result is now a surprisingly simple consequence of the interpolation theory established in Section~\ref{Sec: Interpolation for HDs} and the following stability result for complex interpolation originally due to Sneiberg ~\cite{Sneiberg-Original}, see also \cite[Thm.\ 2.7]{Sneiberg-KaltonMitrea}.

\begin{proposition}
\label{Prop: Sneiberg}
Let $(X_0, X_1)$ and $(Y_0,Y_1)$ be interpolation couples and let $T: X_0 + X_1 \to Y_0 + Y_1$ be a linear operator that for $j=0,1$ restricts to a bounded operator from $X_j$ into $Y_j$. Then
\begin{align*}
 \big \{\theta \in (0,1) \big| \text{ $T: \big[X_0, X_1\big]_\theta \to \big[Y_0, Y_1\big]_\theta$ is an isomorphism} \big \}
\end{align*}
is an open subset of $(0,1)$.
\end{proposition}

In order to apply this result, put $(X_0, X_1) := (\H_D^{2/3}, \H_D^{4/3})$ and $(Y_0, Y_1):= (X_1^*,X_0^*)$. By Theorem~\ref{Thm: Main interpolation theorem} the complex interpolation spaces induced by the couple $(X_0,X_1)$ are
\begin{align}
\label{Eq: Interpolation X0 X1}
 \big[X_0,X_1 \big]_\theta = \H_D^{2 \alpha} \qquad (\theta \in [0,1], \, \alpha = \tfrac{1+\theta}{3}).
\end{align}
In particular, the smallest space $\H_D^{4/3}$ is dense in $\H_D^{2 \alpha}$ for each $\alpha \in [\frac{1}{3}, \frac{2}{3}]$, cf.\ \cite[Thm.\ 4.2.2]{Bergh-Loefstroem}. For these values of $\alpha$ the anti dual spaces $(\H_D^{2 \alpha})^*$ can be naturally embedded into $(\H_D^{4/3})^*$ via restriction of functionals. In virtue of these embeddings $(Y_0, Y_1)$ is an interpolation couple and due to \eqref{Eq: Interpolation X0 X1}, reflexivity of $X_0$, cf.\ Corollary~\ref{Cor: HDs are reflexive}, and duality for complex interpolation \cite[Cor.\ ~4.5.2]{Bergh-Loefstroem} the induced interpolation spaces are
\begin{align*}
 \big[Y_0,Y_1\big]_\theta = (\H_D^{2-2\alpha})^* \qquad (\theta \in [0,1], \, \alpha = \tfrac{1+\theta}{3}).
\end{align*}

Lemma~\ref{Lem: Second Interpolation estimates for j} asserts that \emph{the duality map} $u \mapsto \j(u, \argdot)$ extends by density from $\V$ to a bounded operator $\J: X_0 \to Y_0$ which, owing to the symmetry of $\j$, maps $X_1$ boundedly into $Y_1$. Hence, by Sneiberg's stability result
\begin{align*}
 I:= \big \{\alpha \in (\tfrac{1}{3},\tfrac{2}{3}) \big| \text{ $\J: \H_D^{2 \alpha} \to (\H_D^{2-2\alpha})^*$ is an isomorphism} \big \}
\end{align*}
is an open subset of $(\frac{1}{3},\frac{2}{3})$. Thanks to the Lax-Milgram lemma $\frac{1}{2} \in I$. Hence, there exists $\eps_0 \in (0, \frac{1}{6})$ such that $[\frac{1}{2}-\eps_0, \frac{1}{2}+\eps_0] \subseteq I$.

Now, let $\alpha \in [\frac{1}{2},\frac{1}{2} + \eps_0]$ and take $u \in \dom((1-\Delta_\V)^\alpha) \subseteq \V$. A reformulation of Lemma~\ref{Lem: First Interpolation estimates for j} is that $\J u=\j(u,\cdot)$ is a bounded conjugate linear functional on $\H_D^{2-2\alpha}$ with norm not exceeding the graph norm of $u$. Due to $\alpha \in I$ it follows 
\begin{align*}
 \|u\|_{\H_D^{2\alpha}} \lesssim \|\J u\|_{(\H_D^{2-2\alpha})^*} \lesssim  \|u\|_{\dom((1-\Delta_\V)^\alpha)},
\end{align*}
i.e.\ $\dom((1-\Delta_\V)^\alpha) \subseteq \H_D^{2 \alpha}$ with continuous inclusion. To see that for $\alpha$ close enough to $\frac{1}{2}$ we have in fact equality, first recall from \eqref{Eq: Fractional power domains under square root} that $\H_D^{2 \alpha} = \dom((1-\Delta_\V)^\alpha)$ holds if $\alpha \in (\frac{1}{4},\frac{1}{2}]$. Combining this with the previously established continuous inclusion we see that
\begin{align*}
 \Id: \dom((1-\Delta_\V)^\alpha) \to \H_D^{2 \alpha}
\end{align*}
is bounded if $\alpha \in (\frac{1}{4}, \frac{1}{2} + \eps_0]$ and an isomorphism if $\alpha \in (\frac{1}{4},\frac{1}{2}]$. Since the domains of the fractional powers of $1-\Delta_\V$ interpolate according to Proposition~\ref{Prop: BIP Lemma}, we can re-apply Proposition~\ref{Prop: Sneiberg} to obtain an $\eps < \eps_0$ such that $\Id: \dom((1-\Delta_\V)^\alpha) \to \H_D^{2 \alpha}$ is an isomorphism for all $\alpha \in [\frac{1}{2}, \frac{1}{2} + \eps)$. This establishes our ultimate goal \eqref{Eq: Last goal} and thereby completes the proof of Theorem~\ref{Thm: Main Result}. \hfill $\square$

\section{Elliptic Systems}
\label{Sec: Elliptic Systems}

\noindent In this section we extend Theorem~\ref{Thm: Kato} to coupled systems of elliptic operators on $\Omega$ of the form
\begin{align*}
\begin{array}{c c c}
 (\SysA u)_1 &=& - \displaystyle{\sum_{m,n = 1}^d \sum_{k=1}^N \partial_m(\mu_{m,n}^{1,k} \partial_n u_k)} \\
 \vdots& \vdots &\vdots \\
 (\SysA u)_N &=&- \displaystyle{\sum_{m,n = 1}^d \sum_{k=1}^N \partial_m(\mu_{m,n}^{N,k} \partial_n u_k)}
\end{array}
\end{align*}
with coefficients $\mu_{m,n}^{j,k} \in \L^\infty(\Omega)$ and mixed boundary conditions with possibly different Dirichlet parts $D_j$ for each component $u_j$. We assume that each pair $(\Omega, D_j)$ satisfies Assumption~\ref{Ass: General geometric assumption on Omega}, i.e.\ that the following holds.

\begin{assumption}
\label{Ass: General geometric assumption for Systems}
\begin{enumerate}
\itemsep3pt
 \item The domain $\Omega \subseteq \IR^d$, $d \geq 2$, is a non-empty, bounded $d$-set.
 \item The Dirichlet parts $D_j$, $1\leq j \leq N$, are closed subsets of $\bd \Omega$ and each of them is either empty or a $(d-1)$-set. 
 \item Around every point of the closure of $\bd \Omega \setminus \bigcap_{j=1}^N D_j = \bigcup_{j=1}^N \bd \Omega \setminus D_j$ there exists a bi-Lipschitz coordinate chart as in Assumption~\ref{Ass: General geometric assumption on Omega}.
\end{enumerate}
\end{assumption}

To define an appropriate form domain for $\SysA$ first take $\V_j$, $1 \leq j \leq N$, as the closure of $\C_{D_j}^\infty(\Omega)$ under the norm $\|u_j\|_{\V_j} := (\int_\Omega \abs{u_j}^2 +~ \abs{\nabla u_j}^2)^{1/2}$ and then put $\SysV:= \prod_{j=1}^N \V_j = \prod_{j=1}^N \H_{D_j}^1(\Omega)$. Here, the second equality is due to Corollary~\ref{Cor: Form domain equals HD1}. Similar to Subsection~\ref{Subsec: The Elliptic Operator} we identify $\SysA$ with the maximal accretive operator on $\L^2(\Omega)^N$ associated to the elliptic sesquilinear form
\begin{align*}
 \a: \SysV \times \SysV \to \IC, \quad \a(u,v) = \sum_{m,n = 1}^d \sum_{j,k=1}^N \int_\Omega \mu_{m,n}^{j,k}\partial_n u_k \cdot \partial_m \cl{v_j},
\end{align*}
and make the following assumption.

\begin{assumption}
\label{Ass: Ellipticity for systems}
There exists some $\lambda > 0$ such that the following G\aa{}rding inequality holds:
\begin{align*}
 \Re (\a(u,u)) \geq \lambda \sum_{j=1}^N \|\nabla u_j\|_{\L^2(\Omega; \IC^d)}^2 \qquad (u \in \SysV)
\end{align*}
\end{assumption}

 Here, and throughout, we write $u_j$, $1 \leq j \leq N$, for the component functions of $u \in \L^2(\Omega)^N$. This setup for elliptic systems has been previously studied e.g.\ in \cite{Rehberg-Jonsson}. For a survey on regularity results for elliptic systems with rough coefficients, see e.g.\ \cite{Mazya-survey}.

For $1 \leq j \leq N$ let $\Delta_{\V_j}$ be the weak Laplacian with form domain $\V_j$, cf.\ Subsection~\ref{Subsec: The Elliptic Operator}. For the choice $\mu_{m,n}^{j,k} = \delta_{m,n} \delta_{j,k}$, where $\delta$ is Kronecker's delta, the sesquilinear form $\a$ becomes
\begin{align*}
 \SysV \times \SysV \to \IC, \quad (u,v) \mapsto \sum_{j=1}^N \int_\Omega \nabla u_j \cdot \nabla \cl{v_j}
\end{align*}
and it can easily be checked that the associated operator is the negative componentwise Laplacian
\begin{align*}
 -\Delta_\SysV = \mathrm{diag}(-\Delta_{\V_1},\ldots, -\Delta_{\V_N}) \quad \text{on} \quad \dom(-\Delta_\SysV) = \prod_{j=1}^N \dom(-\Delta_{\V_j}).
\end{align*}

The subsequent theorem solves the Kato Square Root Problem for the general coupled elliptic system $\SysA$. The proof relies again on the reduction results in \cite{Laplace-Extrapolation-Implies-Kato}. The key observation is the following decoupling property: It suffices to work with the \emph{diagonal} system $-\Delta_\SysV$ instead of the general \emph{coupled} system $\SysA$. But all properties of the system $-\Delta_\SysV$ can be obtained from the previous sections by coordinatewise considerations.

\begin{theorem}
\label{Thm: Kato for systems}
Under Assumption~\ref{Ass: General geometric assumption for Systems} the domain of $\SysA^{1/2}$ coincides with the form domain $\SysV$ and
\begin{align*}
 \|\SysA^{1/2}u\|_{\L^2(\Omega)^N} \simeq \|(\nabla u_j)_{j=1}^N\|_{\L^2(\Omega; \IC^d)^N}  \qquad (u \in \dom(\SysA^{1/2})).
\end{align*}
\end{theorem}

\begin{proof}
We have stated Theorem~\ref{Thm: Kato from Laplace} for a single equation only but in fact this result is proved in \cite{Laplace-Extrapolation-Implies-Kato} for elliptic systems $\SysA$ as defined above. The assumptions are the same upon the obvious modifications such as replacing $\V$ by a function space that contains $\C_c^\infty(\Omega)^N$ and that is closed under the norm $u \mapsto \sum_{j=1}^N (\int_\Omega \abs{u}^2 +~ \abs{\nabla u}^2)^{1/2}$, and $\Delta_\V$ by $\Delta_\SysV$. Since by Assumption~\ref{Ass: General geometric assumption for Systems} the domain $\Omega$ is a $d$-set and $\partial \Omega$ is a $(d-1)$ set, cf.\ Remark~\ref{Rem: Remark on geometric assumptions}, this theorem then gives $\dom(\SysA^{1/2}) = \SysV$ with the inhomogeneous estimate 
\begin{align}
\label{Eq: Inhomogeneous estimate Systems}
 \|(1+\SysA)^{1/2}u\|_{\L^2(\Omega)^N} \simeq \|(u_j)_{j=1}^N\|_{\L^2(\Omega)^N} + \|(\nabla u_j)_{j=1}^N\|_{\L^2(\Omega; \IC^d)^N}  \qquad (u \in \dom(\SysA^{1/2})),
\end{align}
provided we can take care of the following.
\begin{enumerate}
\itemsep3pt
 \item[($\SysV$)] The form domain is stable under multiplication by smooth scalar valued functions in the sense that $\varphi \SysV \subseteq \SysV$ holds for each $\varphi \in \C_c^\infty(\IR^d)$. Moreover, it has the \emph{$\H^1$extension property}, i.e.\ there exists a bounded operator $\E_{\SysV}: \SysV \to \H^1(\IR^d)^N$ such that $\E_{\V} u = u$ holds a.e.\ on $\Omega$ for each $u \in \SysV$.
 \item[($\alpha$')] There exists an $\alpha \in (0, 1)$ such that the complex interpolation space $[\L^2(\Omega)^N, \SysV]_\alpha$ coincides with $\H^\alpha(\Omega)^N$ up to equivalent norms.
 \item[(E')] For the \emph{same} $\alpha$ as above $\dom((-\Delta_\SysV)^{1/2 + \alpha/2}) \subseteq \H^{1+\alpha}(\Omega)^N$ holds with continuous inclusion.
\end{enumerate}
In Section~\ref{Sec: Main results} we have seen that for each $1 \leq j \leq N$ the space $\V_j$ is stable under multiplication by smooth scalar valued functions and that it has the $\H^1$ extension property. Thus, ($\SysV$) follows. To establish ($\alpha$') and (E') first note that if $\Re(\alpha) > 0$ then the Balakrishnan Representation \eqref{Balakrishnan Representation} readily yields
\begin{align*}
 (-\Delta_\SysV)^\alpha = \mathrm{diag}((-\Delta_{\V_1})^\alpha,\ldots, (-\Delta_{\V_N})^\alpha) \quad \text{on} \quad \dom((-\Delta_\SysV)^\alpha) = \prod_{j=1}^N \dom((-\Delta_{\V_j})^\alpha).
\end{align*}
Thanks to Theorem~\ref{Thm: Main Result} each $-\Delta_{\V_j}$ satisfies ($\alpha$) and (E) from Section~\ref{Sec: Main results} not only for a single $\alpha$ but for all $\alpha$ in some open interval with lower endpoint $0$. Hence, ($\alpha$) and (E) are met simultaneously by all $-\Delta_{\V_j}$, $1 \leq j \leq N$, if $\alpha>0$ is sufficiently small. This gives ($\alpha$') and (E').

Finally, the required homogeneous estimate can be deduced from \eqref{Eq: Inhomogeneous estimate Systems} by literally the same arguments as in the proof of Lemma~\ref{Lem: Homogeneous Kato is easy}. Note that the results in \cite[Sec.~6]{Hardy-Poincare} carry over to $\IC^N$-valued spaces word by word.
\end{proof}


\end{document}